\documentclass[a4paper,11pt]{article}
\usepackage[margin=0.95in,left=2cm, right=2cm]{geometry}
\usepackage{authblk}
\usepackage{tgpagella,eulervm}
\usepackage{graphicx}
\usepackage{amssymb,amsmath,amsthm}
\usepackage{kotex}
\usepackage{soul}
\usepackage[toc,page]{appendix}
\usepackage{thmtools}
\usepackage{hyperref}
\usepackage{cleveref}
\usepackage[shortlabels]{enumitem}
\usepackage{dotlessi}
\usepackage{tikz}
\usepackage{pgfplots}
\usetikzlibrary{pgfplots.groupplots}
\usetikzlibrary{decorations.pathreplacing}
\usetikzlibrary{arrows.meta}
\usepackage{tikz-cd}
\usepackage{stmaryrd}
\usepackage{ifthen}
\usetikzlibrary{calc}
\usepackage{mathtools} %

\usepackage{tocloft}
\usepackage{adjustbox} %
\usepackage{listings} %
\tikzcdset{scale cd/.style={every label/.append style={scale=#1},
    cells={nodes={scale=#1}}}}

\let\amsamp=&
\theoremstyle{definition}
\newtheorem{theorem}{Theorem}[section]

\soulregister\cite7
\soulregister\ref7
\soulregister\Cref7
\soulregister\cref7
\soulregister\eqref7

\declaretheorem[name=Definition, numberlike=theorem]{definition}
\declaretheorem[name=Proposition, numberlike=theorem]{proposition}
\declaretheorem[name=Corollay, numberlike=theorem]{corollary}
\declaretheorem[name=Lemma, numberlike=theorem]{lemma}
\declaretheorem[name=Example, numberlike=theorem]{example}

\declaretheorem[name=Remark, numberlike=theorem]{remark}

\usepackage{xcolor}
\definecolor{darkblue}{rgb}{0.0, 0.0, 0.8}
\definecolor{darkred}{rgb}{0.8, 0.0, 0.0}
\definecolor{darkgreen}{rgb}{0.0, 0.65, 0.0}
\hypersetup{
    colorlinks,
    linkcolor={darkblue},
    citecolor={darkblue},
    urlcolor={darkblue}
}

\newcommand{\ep}{\epsilon}

\newcommand{\ua}{\uparrow}
\newcommand{\da}{\downarrow}

\newcommand{\R}{\mathbb{R}}

\newcommand{\F}{\mathbb{F}}

\newcommand{\Nz}{\mathbb{N}_0}
\newcommand{\N}{\mathbb{N}}
\newcommand{\Z}{\mathbb{Z}}

\newcommand{\0}{\mathbf{0}}

\newcommand{\sbeq}{\subseteq}

\newcommand{\id}{\mathrm{id}}

\newcommand{\Lan}{\mathrm{Lan}}

\newcommand{\norm}[1]{\left\lVert#1\right\rVert}

\newcommand{\barc}{\mathsf{barc}}
\newcommand{\brac}{\barc}

\newcommand{\colim}{\mathrm{colim}}

\newcommand{\binuparrow}{\mathbin{\uparrow}}
\newcommand{\bindownarrow}{\mathbin{\downarrow}}
\newcommand{\vect}{\mathit{vec}_\F}
\newcommand{\Vect}{\mathit{Vec}}

\newcommand{\ignore}[1]{}

\newcommand{\abs}[1]{\left\lvert{#1}\right\rvert}
\newcommand{\dedi}{\mathsf{d}_{\mathsf{Ed}}}
\newcommand{\sfd}{\mathsf{d}}
\newcommand{\calE}{{\mathcal{E}_d}}
\newcommand{\calC}{\mathcal{C}}
\newcommand{\calD}{\mathcal{D}}
\newcommand{\ob}{\mathrm{ob}}

\newcommand{\dist}[1]{\lVert #1 \rVert}
\newcommand{\ones}{\vec{1}}
\newcommand{\dint}{\mathsf{d}_{\mathrm{I}}}
\newcommand{\bott}{\mathsf{d}_{\mathrm{B}}} 
\newcommand{\inj}{\mathrm{inj}}

\newcommand{\edit}{edit}
\newcommand{\cost}{\mathrm{cost}}

\newcommand{\Poset}{{\mathsf{Poset}_d}}

\newcommand{\jsLat}{{{\mathsf{JS}_d}}}

\newcommand{\JsLat}{\jsLat}

\newcommand{\carte}{{\mathsf{Grid}_d}}
\newcommand{\grid}{\mathsf{Grid}_d}

\newcommand{\lmulti}{\left\{\!\left\{}
\newcommand{\rmulti}{\right\}\!\right\}}
\newcommand{\gr}{\mathrm{gr}}
\newcommand{\sym}{\Sigma}

\newcommand{\fr}{\mathrm{fr}}

\newcommand{\Int}{\mathrm{Int}}
\newcommand{\Dgm}{\Int}
\newcommand{\rk}{\mathrm{rk}}

\newcommand{\rank}{\mathrm{rk}}

\newcommand{\faithful}{faithful}
\renewcommand{\subseteq}{\subset}
\newcommand{\joingrid}{join-semigrid}

\newcommand{\persistcolor}{blue}
\newcommand{\pointcolor}{white}

\newcommand{\pointsize}{1.5pt}

\newcommand{\sizeone}{1}
\newcommand{\sizetwo}{1.3}
\newcommand{\sizethree}{1.7}

\newcommand{\x}{b}

\title{Interleaving Distance as a Galois-Edit Distance}

\author[1]{Woojin Kim}
\author[2]{Won Seong}
\date{} 

\affil[1,2]{Department of Mathematical Sciences, KAIST, South Korea}

\begin{document}
\maketitle

\begin{abstract}
The concept of \emph{edit distance}, which dates back to the 1960s in the context of comparing word strings, has since found numerous applications with various adaptations in computer science, computational biology, and applied topology.
By contrast, the \emph{interleaving distance}, introduced in the 2000s within the study of persistent homology, has become a foundational metric in topological data analysis.
In this work, we show that the interleaving distance on finitely presented single- and multi-parameter persistence modules can be formulated as a so-called \emph{Galois-edit distance}. 
The key lies in clarifying a connection between the Galois connection and the interleaving distance, via the established relation between the interleaving distance and free presentations of persistence modules.  In addition to offering new perspectives on the interleaving distance, we expect that our findings will facilitate the study of stability properties of invariants for multi-parameter persistence modules.

As an application of the Galois-edit formulation of the interleaving distance, we present an alternative proof of the well-known bottleneck stability theorem. 
\end{abstract}

\section{Introduction}

The concept of \emph{edit distance} (also called \emph{Levenshtein distance}) originated in the 1960s, first for comparing strings \cite{levenshtein1966binary,wagner1975extension} and later for trees \cite{tai1979tree} in the 1970s.
An edit distance is defined by specifying a collection of admissible edits, each equipped with an associated cost. The distance between two objects is then given by the minimal total cost among all possible sequences of such edits transforming one object into the other. The distance has numerous applications in computational biology, signal processing, and text retrieval (see, e.g., \cite{navarro2001guided}).
 Its applications expanded in the 2010s to the comparison of Reeb graphs \cite{bauer2016edit,bauer2021reeb,di2012reeb,di2016edit}. More recently, in the study of persistent homology, edit distances between filtered simplicial complexes and between integral functions were introduced in \cite{mccleary2022edit,mccleary2025erratum}. These edit distances were employed to establish a stability theorem for persistence diagrams \cite[Theorems 8.4 and 9.1]{mccleary2022edit,mccleary2025erratum}
, providing an alternative to the well-known formulations of the stability theorem \cite{bauer2013induced,chazal2016structure,cohen2007stability}. 

Meanwhile, the \emph{interleaving distance} is a foundational metric in topological data analysis (TDA). 
Originally introduced in 2009 to compare \emph{persistence modules} over the reals \cite{chazal2009proximity}, we briefly elaborate on this concept in a more general framework \cite{bubenik2015metrics,de2017theory}: Given a poset $P$, a persistence module over $P$ (or simply a $P$-persistence module) is a diagram of vector spaces and linear maps indexed by \( P \), formally defined as a functor \( P \to \Vect_\F \), where \( \Vect_\F \) denotes the category of vector spaces over a fixed field \( \F \).
The indexing poset \( P \) is 
often equipped with a 
superlinear family of translations $P\to P$ or a similar structure—for instance, \( \R^d \) with the product order 
and $\ep$-translations $(x_1,\ldots,x_d)\mapsto (x_1,\ldots,x_d)+\ep(1,\ldots,1)$ for $\ep\geq 0$, the setting for \emph{multi-parameter persistence modules} \cite{ carlsson2009theory,lesnick2015theory} (see also \cite{bauer2026algebraic,botnan2022introduction,kim2023persistence} for a comprehensive overview).   For such a poset $P$, by relaxing the notion of natural isomorphism between $P$-persistence modules, one obtains the interleaving distance between $P$-persistence modules. 
In particular, when \(d=1 \), the interleaving distance between 
persistence modules whose structure maps $M_{p\le p'}$ for all pairs $p\leq p'$ in $\R$ have finite rank can be reformulated as the \emph{bottleneck distance} \cite{chazal2009proximity,cohen2007stability}, which is of a more combinatorial nature. This reformulation is made possible by the structural property of such $\R$-persistence modules  \cite{bauer2013induced,chazal2016structure,crawley2015decomposition}.

Other studies on interleaving distances, focusing on aspects such as computation or generalization, include  \cite{bjerkevik2020computing,bjerkevik2021ell,botnan2020relative,botnan2018algebraic,bubenik2017interleaving,curry2014sheaves,de2016categorified,de2018theory,kim2024interleaving,morozov2013interleaving}.
Beyond their established importance,
the following facts point to a potentially deeper connection between edit and interleaving distances:
\begin{itemize}[noitemsep]
\item The Reeb edit distance is the most discriminative  among stable metrics on Reeb graphs~\cite{bauer2021reeb}. 
\item The interleaving distance is the most discriminative among stable metrics on $\R^d$-persistence modules with coefficients in a prime field ~\cite[Corollary~5.6]{lesnick2015theory}.
\item The Reeb interleaving distance is bi-Lipschitz equivalent to the Reeb edit distance \cite{bauer2022quasi,bauer2025tight}.

\item The interleaving distance on finitely presented $\R$-persistence modules is bi-Lipschitz equivalent to the edit distance on persistence diagrams; this follows from  \cite[Theorem 9.1]{mccleary2022edit,mccleary2025erratum} and the bottleneck-interleaving isometry theorem \cite{bauer2013induced,chazal2009proximity,lesnick2015interactive}.
 \end{itemize}
 Taken together, these considerations motivate the search for a suitable notion of edit distance between persistence modules and for its relationship with the interleaving distance between persistence modules.  From another perspective, the fact that this search can potentially yield a practical and tractable characterization of (interleaving-)stable invariants for $\R^d$-persistence modules (e.g. \cite{asashiba2026minimal,bakke2021stability,botnan2022bottleneck,cerri2013betti,clause2022discriminating,fersztand2024harder-stability,landi2018rank,mccleary2022edit,mccleary2025erratum,oudot2024stability,scolamiero2017multidimensional})  has been the foundational motivation for this study.

\paragraph{Our contributions.}
We show that the interleaving distance $\dint$ on finitely presented $\R^d$-persistence modules can be formulated as an edit distance, called the \emph{Galois-edit distance},
via a suitable specification of the objects and edits allowed in the edit sequences (\textbf{\Cref{def:edit distance}} and \textbf{\Cref{thm:dedi=dint}}). 
The key lies in clarifying a connection between the \emph{Galois connection} \cite{ore1944galois}
and the interleaving distance $\dint$, 
via the established relation between 
$\dint$ and free presentations of persistence modules.

The Galois-edit formulation of 
$\dint$ can be viewed as a general method for decomposing a \emph{path} in the metric space $(\mathcal{M},\dint)$ of finitely presented persistence modules, that is, a continuous map $[0,1]\to \mathcal{M}$. In fact, the idea of decomposing such a path into smaller pieces has already been utilized in the establishment of stability of persistence barcodes, and each piece  often corresponds to a Galois-edit. 
However, the idea of decomposing paths into Galois-edits has not been leveraged to formulate the interleaving distance as a ``faithful” edit distance (\textbf{\Cref{rem:in-depth}} and \textbf{\Cref{prop:edit sequence of length two is not optimal}}); here, “faithful” means that a shortest edit sequence between two objects achieving the edit distance can be arbitrarily long.\footnote{Note that any distance $d$ on a set $X$ can trivially be viewed as a non-faithful edit distance by defining the edit cost between any two points $a,b\in X$ as 
$d(a,b)$.}  In this light, this work provides a unified theoretical framework that encompasses “path-decomposition” techniques. 

From another perspective, as the unit of our edit is a Galois connection between subposets of the indexing posets of persistence modules, our results provide a novel order-theoretic viewpoint on 
$\dint$, as an alternative to the combinatorial viewpoint encoded in the bottleneck–interleaving isometry \cite{bauer2013induced,chazal2009proximity,lesnick2015theory}.
In addition, our findings 
provide a simple characterization of interleaving-stable invariants
for multi-parameter persistence modules (\textbf{\Cref{cor:one stability implies another,cor:single-edit is enough}}). 
 Our Galois-edit formulation of $\dint$
 is particularly advantageous as it allows us to harness \emph{at the same time} two key tools when establishing stability of invariants: (1) Rota's theory of M\"obius inversion \cite{rota1964foundations} and (2) the well-known maximality property of the edit distance (\textbf{\Cref{prop:edit_is_universal}} and its subsequent paragraph). We 
 demonstrate the utility of both tools in our alternative proof of the bottleneck stability theorem (\textbf{\Cref{thm:bottleneck stability}} and \textbf{\Cref{rem:usefulness of Galois connection}).}

In establishing the Galois-edit distance formulation of $\dint$, a critical step is to show that the Galois-edit distance is bounded above by $\dint$, for which \textbf{\Cref{thm:Galois_edit_is_presentation_edit}} plays a central role. Roughly speaking, \Cref{thm:Galois_edit_is_presentation_edit} reduces the problem of finding a Galois-edit sequence between persistence modules to one that is reminiscent of an optimal transportation problem; 
namely, finding a Galois-edit between persistence modules is equivalent to specifying their presentations together with a monotone-map–induced bijection between the union of the generators and relations of one presentation and the union of those of the other.

We remark that some of our proofs are obtained by adapting ideas from the literature (see \Cref{rem:in-depth}). Nevertheless, to the best of our knowledge,
\begin{itemize}
\item the Galois-edit formulation presented in \Cref{thm:dedi=dint} is new, as is the resulting proof of \Cref{thm:bottleneck stability};
\item \Cref{thm:Galois_edit_is_presentation_edit} and its proof are new.
\end{itemize}

\paragraph{Other related works.} 
Different approaches exist for reformulating the interleaving distance between persistence modules. Below, we situate our work in the context of the two most closely related studies:
\begin{itemize}[leftmargin=1em]
\item Bjerkevik and Lesnick view the interleaving distance for finitely presented $\R^d$-persistence modules as the $\infty$-\emph{presentation distance}-- a type of an edit distance (but not faithful) -- and extend it and its associated stability results to the $p$-presentation setting (i.e., the $\ell^p$ analogue setting) for $p\in [1,\infty)$
\cite[Definition 3.2 and Theorem 3.7]{bjerkevik2021ell}. More specifically, they provide
an alternative proof of the stability of the fibered barcode \cite{cerri2013betti,landi2018rank} that extends to the $\ell^p$ analogue setting \cite[Theorem 1.7~(iii)]{bjerkevik2021ell}.

Our notion of Galois-edit refines the notion of \emph{presentation matrices with the same underlying matrix} in \cite{bjerkevik2021ell}; namely, a Bjerkevik--Lesnick edit (i.e., presentation matrices with the same underlying matrix) and its cost amount to a (possibly long) sequence of our Galois-edits and its cost; this follows from  \Cref{prop:edit sequence of length two is not optimal} and the fact that the $\infty$-presentation, Galois-edit, and  interleaving distances all coincide (\cite[Theorem 1.7 (i)]{bjerkevik2021ell} and \Cref{thm:dedi=dint}).

\item G\"ulen and McCleary introduce a novel interpretation of the interleaving distance via a \emph{pair} of Galois connections between indexing posets of persistence modules (see \Cref{appendix:comparison with gulen's work}); this interpretation can also be viewed as a (non-faithful) edit distance formulation of the interleaving distance.\footnote{To the best of our knowledge, the use of Galois connections to construct edit distances was initiated in \cite{mccleary2022edit}, although the use was not explicit. The use was further extended to a different setting in \cite{gulen2025ell}.} Their interpretation 
gives rise to an alternative proof of the bottleneck stability theorem \cite[Theorem 7.1]{gulen2022galois}.

An important difference between G\"ulen--McCleary's and our proofs of the bottleneck stability theorem is that they view the persistence diagram as the M\"obius inversion of the birth–death function, whereas we view it as the M\"obius inversion of the rank function, the more classical viewpoint~\cite{cohen2007stability}.
\end{itemize}

The following works are also related to our work to some extent; however, their motivations are orthogonal to our study.

In \cite{bauer2022quasi,bauer2025tight}, Bauer, Bjerkevik, and Fluhr thoroughly compare the Reeb edit distance with
other metrics on Reeb graphs.

In \cite{chacholski2021realisations}, Chacholski, Jin, and Tombari generalize the notions of finitely presented persistence modules and their indexing posets, aiming to exploit discrete approximations of the indexing posets to capture the homological properties of the modules. In particular, the notion of a finitely presented persistence module is generalized to that of a \emph{tame functor} (also referred to as a \emph{discretisable functor}). 
The authors investigate the properties of tame functors and the associated category, including its model structure, and study the computation of minimal resolutions for $\Vect_\F$-valued tame functors.
Further investigations into the model structures of the category of tame functors (along with their abelian structures) are presented in a subsequent work \cite{chacholski2023abelian}.

In \cite{blanchette2023exact}, Blanchett, Br\"ustel, and Hanson establish an adjunction, called \emph{extension} and \emph{contraction} between the categories of $P$- and $Q$-persistence modules---where $P$ is finite and $Q$ is infinite---and thereby lift the study of exact structures on the category of $P$-persistence modules to that of $Q$-persistence modules. A central case is when $Q$ is a product of totally ordered posets and $P$ is a \emph{finite aligned subgrid} of $Q$; this setup subsumes one of our central setups, namely Definition~\ref{def:differentclassesforpersistencemoduledomain}~\ref{def:1Dmorphism}. %

In \cite{aoki2025preservation}, Aoki and Tada utilize Galois connections to gain a deeper understanding of the \emph{interval resolution} \cite{asashiba2023approximation} of persistence modules. In particular, they primarily focus on a certain class of Galois connections between posets $P$ and $Q$ that induce interval-decomposability-preserving adjoint functors between the categories of $P$- and $Q$-persistence modules, and thereby they introduce a technique for transforming a given finite poset into a smaller one with the same interval resolution global dimension. %
In particular, the authors
{pay special attention to} an adjunction between the categories of finitely presented $P$- and $Q$-persistence modules when $Q$ is an \emph{interior system}
of $P$, i.e., $Q$ is a full subposet of $P$ and the inclusion $Q\hookrightarrow P$ admits a right adjoint 
\cite[Definition 2]{erne1993primer} \cite[Section 3]{aoki2025preservation}. 
On the other hand, the adjunction pair in our Galois-edit is a generalization of an adjunction induced by 
a \emph{closure system}; see the paragraph after \Cref{rem:gisanembedding}. %

\paragraph{Acknowledgements.} WK thanks Amit Patel for a helpful conversation. 
This research was supported by the National Research Foundation of Korea (NRF) grant funded by the Korea government(MSIT) (RS-2025-00515946).

\paragraph{Organization.} 
In \Cref{sec:preliminaries}, we review basic concepts concerning posets and  persistence modules.
In \Cref{sec:Edit distance for persistence modules,sec:interleaving as edit}, we reformulate the interleaving distance as the Galois-edit distance, and establish a novel connection between Galois-edits and presentations of persistence modules.
In \Cref{sec:functoriality}, we establish a proof of the bottleneck stability theorem, using Galois connections and maximality of an edit distance as key tools. In \Cref{sec:discussion}, we present concluding remarks.

\Cref{appendix:omitted proofs,appendix:another proof}  contain proofs omitted from the main text.
In \Cref{appendix:two constructibility}, we compare the notion of finite presentability with the notion of constructibility from \cite{gulen2022galois}.
In \Cref{appendix:comparison with gulen's work}, we compare the Galois-edit distance with other edit(-type) distances in applied topology. We also compare our Galois-edit formulation of the interleaving distance (\Cref{thm:dedi=dint}) with previously known connections between the interleaving distance and Galois connections.
In \Cref{appendix:optimal sequence can be long}, we show that a Galois-edit sequence between two persistence modules achieving the Galois-edit distance can be arbitrarily long.

\section{Preliminaries}\label{sec:preliminaries}

In this section, we review basic concepts from order theory, category theory, and the algebraic foundations of persistence modules. 

\subsection{Galois connections between posets}

 In this section, we review basic concepts and results on posets, lattices, and Galois connections.

\label{sec:posets}
A \textbf{preordered set} is a set equipped with a preorder, i.e., a reflexive and transitive relation on the set. 
A \textbf{partially ordered set} (\textbf{poset}) is a preordered set 
whose preorder is antisymmetric.
A poset $P=(P,\leq)$ is \textbf{totally ordered} if $a\le b$ or $b\le a$ for any $a,b\in P$. Given any $a\in P$, we define  $a\binuparrow$ to be the set of all points $b\in P$ such that $a\le b$. Similarly, we define  $a\bindownarrow$ to be the set of all points $b\in P$ such that $a\ge b$. For any $A\subset P$, we define $A\binuparrow=\bigcup_{a\in A}a\binuparrow$ and $A\bindownarrow=\bigcup_{a\in A} a\bindownarrow$.
A \textbf{bottom} of $P$ is an element $\bot \in  P$ such that $\bot\le a$ for all $a \in P$.
A \textbf{top} of $P$ is an element $\top \in P$ such that $\top \ge a$ for all $a\in P$. 
When the clarification is needed, we write $\bot_P$ and $\top_P$ for the bottom and the top of $P$, respectively.
A function of posets $f:P\to Q$ is \textbf{monotone} if $f(p)\le f(q)$ holds whenever $p\le q \in P$. 
For any poset $P$, let $P_{\bot}$ denote the poset on the disjoint union $P\sqcup \{\bot\}$ 
, with the extended partial order satisfying $\bot\le p$ for all $p\in P$. 
Note that for any monotone function $f:P\to Q$, we have a canonical morphism $f_\bot : P_\bot \to Q_\bot$ such that $f_\bot(\bot_P)=\bot_Q$, i.e., $(-)_\bot$ is an endofunctor on the category of posets and monotone functions. %

The \textbf{meet} of a subset $A\sbeq P$ is the greatest lower bound of $A$ and denoted by $\bigwedge A$. Dually, the \textbf{join} of $A$ is the least upper bound of $A$ and denoted by $\bigvee A$. 
A poset is a \textbf{join-semilattice} (resp. \textbf{meet-semilattice}) if the poset has all binary joins (resp. binary meets).
 A poset is a \textbf{bounded join-semilattice} (resp. \textbf{bounded meet-semilattice}) if every finite subset has a join (resp. meet). This includes the empty join $\bigvee \emptyset$ (resp. meet), which is the bottom (resp. top). A poset $P$ is a \textbf{lattice} if %
 {$P$ is both a join- and a meet-semilattice.}
It is not difficult to see that any finite lattice 
contains both a bottom and a top.

An \textbf{embedding of a poset}, or \textbf{order embedding}, is a monotone function that induces an isomorphism of posets between its domain and its image. In this paper, any subposet of a poset is assumed to be a \emph{full} subposet.

For any integer $d\ge 1$, a \textbf{$d$-grid} (or simply a \textbf{grid}) is a poset of the form
$X_1 \times \cdots \times X_d$, where each $X_i \subset \R$ is finite, with
the order defined by
\(
  (x_1,\dots,x_d) \le (y_1,\dots,y_d)
  \quad\text{iff}\quad
  x_i \le y_i \text{ in } \R \text{ for all } i.
\)
Any grid is finite by definition. %
Let $[d]$ denote the set $\{1,2,\ldots,d\}$. For $i\in [d]$, let $\pi_i:\R^d\to\R$ denote the projection
$(a_1,\ldots,a_d)\mapsto a_i$. 

\begin{remark} \label{rem:smallest grid}
For any finite $P\subset \R^d$, the smallest grid containing $P$ is 
    $\displaystyle \prod_{i\in [d]} \pi_i(P).$
\end{remark}

A \textbf{Galois connection} between two posets $P$ and $Q$ is a pair of monotone functions $f : P \rightarrow Q$ and $g:Q\rightarrow P$ satisfying 
 $f(p)\le q \iff p\le g(q)$ for all $p\in P$ and $q\in Q$ \cite{ore1944galois,rota1964foundations}. 
The map 
 $f$ is a \textbf{left adjoint} of $g$, and $g$ is a \textbf{right adjoint} of $f$, and we write $f:P\leftrightarrows{} Q : g$ or $f\dashv g$.
 Note that if either left or right adjoint is given, the other adjoint is unique. 
 
 For the following remark, we recall that any poset $P$ can be regarded as a category whose objects are the elements of $P$, with a unique morphism $p\to q$ whenever $p\le q$ in $P$. 

\begin{remark}
\label{remark:galois basics}

  \begin{enumerate}[label=(\roman*)]
  \item  When $f\dashv g$, the \emph{triangle identities for the adjunction} give two equations 
  $f\circ g\circ f = f$ and $g\circ f \circ g = g$: see, for example,  \cite[IV.5 Theorem 1]{mac2013categories}. \label{remark:triangle_identity}
 
 \item   Let $f_1:P\leftrightarrows Q:g_1$ and $f_2:Q\leftrightarrows R : g_2$ be Galois connections between posets. Then the composition $f_2\circ f_1: P\leftrightarrows R:g_1\circ g_2$ is also a Galois connection \cite[IV.8 Theorem 1]{mac2013categories}.\label{rem:compositionofGalois}

 \item Let $f:P\leftrightarrows Q:g$ be a Galois connection, and 
 let $\iota:f(P)\hookrightarrow Q$  be the inclusion. %
 Then the pairs of morphisms $f:P\leftrightarrows f(P):g|_{f(P)}$ and $\iota : f(P) \leftrightarrows Q : fg$ are Galois connections, and their composition is the given Galois connection $f:P\leftrightarrows Q:g$
 \cite[Section 3.1]{aoki2025preservation}. \label{rem:decomposition_of_Galois_connection}
\end{enumerate}
    
\end{remark}

Necessary and sufficient conditions for the existence of left and right adjoints are as follows.

\begin{proposition}[{\cite[Lemma 4.6.1]{riehl2017category}}]\label{prop:existence of adjoint}
\begin{enumerate}[label=(\roman*)]\label{item:existence of adjoint1}
   \item  Let $f:P\to Q$ be a monotone function and suppose that for any $q\in Q$, the subset $f^{-1}(q\da)\subset P$ has a maximum in $P$. Then, $f$ admits the right adjoint $g:Q\to P$, which is given by $g(q)=\max f^{-1}(q\da)$. Conversely, if $f$ admits a right adjoint, then $f^{-1}(q\da)$ has a maximum in $P$, and the right adjoint $g$ is given by $g(q)=\max f^{-1}(q\da)$.\label{item:existence of right adjoint}

   \item \label{item:existence of adjoint2}
    Let $P\leftarrow Q:g$ be a monotone function and suppose that for any $p\in P$, the subset $g^{-1}(p\ua)\subset Q$ has a minimum in $Q$. Then, $g$ admits 
   the left adjoint $f:P\to Q$, which is given by $f(p)=\min g^{-1}(p\ua)$. Conversely, if $g$ admits a left adjoint, then $g^{-1}(p\ua)$ has a minimum in $Q$,  and the left adjoint $f$ is given by $f(p)=\min g^{-1}(p\ua)$.\label{item:existence of left adjoint} 
\end{enumerate}   
\end{proposition}

Some direct consequences of the preceding proposition are as follows.

\begin{remark} \label{productofadjoints}
    \begin{enumerate}[label=(\roman*)]
    \item Given a Galois connection $f:P\leftrightarrows Q:g$, \Cref{prop:existence of adjoint}~\ref{item:existence of right adjoint} implies $f\circ g \leq \id_Q$, whereas \Cref{prop:existence of adjoint}~\Cref{item:existence of adjoint2} implies $g\circ f \geq \id_P$.\label{item:fg is deflating and gf is inflating}
        \item   Given a Galois connection  $f: P\leftrightarrows{} Q:g$, the left adjoint $f$ preserves all joins in $P$, i.e., for every $A\subset P$ such that the join $\bigvee A$ exists in $P$, the join $\bigvee f(A)$ exists in $Q$ and $f(\bigvee A)=\bigvee f(A)$
        . This implies that if $\bot_P$ exists, then $\bot_Q$ exists in $Q$ and $f(\bot_P)=\bot_Q$, as the empty join is the bottom element. 
       In contrast, the right adjoint $g$ preserves all meets and thus the top. This is a direct corollary of \cite[V.5 Theorem 1]{mac2013categories}.\label{productofadjoints0}
        
        \item Here is a partial converse of the preceding statement. If $P$ is a finite (bounded join-semi)lattice, then a monotone function $f:P\to Q$ has a right adjoint if and only if $f$ preserves all joins. 
         Dually, if $Q$ is a finite (bounded meet-semi)lattice, then a monotone function $P\leftarrow Q:g$ has a left adjoint if and only if $g$ preserves all meets. This is a direct corollary of \cite[V.6 Theorem 2]{mac2013categories}. 
        \label{item:productofadjoints1}

        \item Let $P$ be a finite join-semilattice. For any monotone function $f:P\to Q$ that preserves all binary joins, by \Cref{item:productofadjoints1}, the function $f_\bot: P_\bot \to Q_\bot$ admits a right adjoint. The right adjoint is called the \textbf{floor function}, denoted by $\lfloor - \rfloor^Q_P$. For $p\in P_{\bot}$, we will abbreviate $\lfloor - \rfloor^Q_P(p)$ by $\lfloor p\rfloor^Q_P$. \footnote{ 
        Concepts closely related to our notion of the floor function have appeared in the literature. For example, in \cite[Section 2.2]{chacholski2021realisations}, the authors introduce the notion of \emph{transfer} which is more general than that of the floor function $\lfloor - \rfloor_P^Q$: 
        The floor function $\lfloor - \rfloor_P^Q$ that is associated with a monotone function $f:P\to Q$  requires that $P$ be a \emph{finite} join-semilattice and that $f$ 
        preserve all binary joins, whereas the transfer of $f:P\to Q$ only requires that $P$ be a (not necessarily finite) join-semilattice and that $f$ be a function of \emph{finite-type}, i.e., for all $q\in Q$, $\{p\in P:f(p)\le q\}$ is finite;  even join-preserving property is not required, although the examples considered by the author satisfy it. %
        In an earlier work \cite{erne1993primer}, a full subposet $P$ of a poset $Q$ for which the inclusion $P\hookrightarrow Q$ has a right adjoint is called an \emph{interior system}. 
        } 
        \label{item:productofadjoints2}

\end{enumerate}
\end{remark}

\subsection{Persistence modules, barcodes, and finitely presentedness}

In this section, we review basic notions of persistence modules, their barcodes, and their free presentations. 

Throughout this paper, $\F$ denotes a fixed field. By $\vect$, we will denote the category of finite-dimensional $\F$-vector spaces and $\F$-linear maps. Recall that any poset $P$ can be seen as a category where the object set is $P$ and there is a unique morphism $p\to q$ whenever $p\le q$ in $P$. A \textbf{persistence module over a poset $P$} (or simply a \textbf{$P$-persistence module}) is a functor $M:P\to \vect$. 
For any $P$-persistence modules $M$ and $N$, their direct sum $M\oplus N$ is defined pointwise. A $P$-persistence module $M$ is \textbf{trivial} or \textbf{zero} if $M(p)=0$ for all $p\in P$, and we write $M=0$.  Let $\cong$ denote a natural isomorphism between $P$-persistence modules. 

An \textbf{interval} of a poset $P$ is a nonempty subset $I\subseteq P$ such that: (i)     
 If $p,q\in I$ and $p\leq r\leq q$, then $r\in I$, and (ii)
 $I$ is \textbf{connected}, i.e., for any $p,q\in I$, there is a sequence $p=p_0,
		p_1,\cdots,p_\ell=q$ of elements of $I$ with $p_i$ and $p_{i+1}$ comparable for $0\leq i\leq \ell-1$.
Given any interval $I$, the \textbf{interval module} $\F_I$ is the $P$-persistence module, with
\begin{equation}\label{eq:interval module} \F_I(p) := \begin{cases} \F & \mathrm{if \ } p\in I\\
0 & \mathrm{otherwise}
\end{cases},
\hspace{10mm}
{\F_I}(p\leq q) := \begin{cases} \id_\F & \mathrm{if \ } p\leq q\in I\\
0 & \mathrm{otherwise}\end{cases}.\end{equation}

Every interval module is non-isomorphic to a direct sum of its two nontrivial submodules \cite[Proposition 2.2]{botnan2018algebraic}. 
\label{nom:barcode} 
A $P$-persistence module $M$ is {called} \textbf{interval-decomposable} if it is isomorphic to a direct sum of interval modules. For example, every $\R$-persistence module is interval-decomposable \cite{crawley2015decomposition}.

A \textbf{multiset} is a collection of objects (called \emph{elements}) in which elements may occur more than once. We call the number of instances of an element the \textbf{multiplicity} of the element. 
\begin{definition}[{\cite{botnan2018algebraic,carlsson2004persistence}}]\label{def:barcode}The \textbf{barcode} of an interval decomposable $P$-persistence module $M\cong \bigoplus_{j\in J}\F_{I_j}$ is defined as the multiset $\barc(M):=\lmulti I_j:j\in J\rmulti$. 
\end{definition}

Let $f:P\to Q$ be a monotone function, where $P$ is a finite poset. Since the category $\vect$ admits all finite colimits, any $P$-persistence module $M$ admits a left Kan extension along $f$, and is denoted by $\Lan_f M$. Namely, the $Q$-persistence module $\Lan_f M$ can be defined as
    $(\Lan_f M)_q = \colim (M|_{P_{f\leq q}})$ for all $q\in Q$ (where $P_{f\leq q}:=\{p\in P : f(p)\leq q\}$) and the map $(\Lan_f M)_{q\le q'}$
    for $q\leq q'$ in $Q$ is 
    the canonical map given by the universal property of colimit \cite[X.3 Theorem 1]{mac2013categories}. 
    We note the isomorphism $\Lan_f M\circ f \cong M$.

\begin{definition}\label{KanExtensionOfPersistenceModules}
\begin{enumerate}[label=(\roman*)]
\item     Given a 
    $P$-persistence module $M$, we define $M_\bot: P_\bot\to \vect$ by trivially extending $M$ to $P_\bot$: Let $M_\bot(\bot)=0$ (the zero space) and $M_\bot(\bot\le p)=0$ (the zero morphism) for all $p\in P$. \label{item:bottom extension}
\item Let $P\subset \R^d$ be a finite subposet.
An $\R^d$-persistence module $M$ is \textbf{finitely presented over $P$} (or simply \textbf{$P$-presented}) if $M$ is the left Kan extension of the restriction $M|_P:P\rightarrow \vect$ along the inclusion. We call $M$  \textbf{finitely presented} if $M$ is $P$-presented for some finite $P\subset \R^d$. An equivalent definition is given in \Cref{def:alternative definition of finitely presented over P}.\footnote{More general notions than finitely presentedness appear in \cite[Definition 8.1]{chacholski2021realisations} and \cite[Section 2.2]{aoki2025preservation}.} 
\label{item:constructible} 
\end{enumerate}
\end{definition}

In the appendix,  we show  that \Cref{KanExtensionOfPersistenceModules}~\ref{item:constructible} is a suitable adaptation of the notion of \emph{constructibility} from \cite{gulen2022galois} to $\R^d$-persistence modules (\Cref{prop:constructibility_extends_original}).

\begin{remark} \label{rem:constructibility}%

Let $P$, $Q$ be finite subposets in $\R^d$ with $P\subset Q$. If an $\R^d$-persistence module $M$ is $P$-presented, then $M$ is also $Q$-presented \cite[Proposition 8.2 (2)]{chacholski2021realisations}. %
\label{item:constructible over a larger poset}

\end{remark}

\subsection{Special colimits}
In this section, we review some propositions on colimits of functors, which will be used in a later section. 

\paragraph{Final functors and colimits.} The notion of a \emph{final functor} \cite[IX.3]{mac2013categories} specializes to monotone functions as follows (viewing monotone functions as functors): We call a monotone function $f\colon P \to Q$ \textbf{final} if, for every $q \in Q$, the subposet $P_{f \ge q} := \{ p \in P : f(p) \ge q \}$ of $P$ is connected. As a direct implication of \cite[IX.3 Theorem 1]{mac2013categories}, we have: 
\begin{proposition}
\label{lem:final_functor_preserves_colimit}
 Let $f:P\to Q$ be a monotone function between finite posets, and $M$ a $Q$-persistence module.
 If $f$ is final, then the natural map
 $ \colim (M\circ f) \to \colim\, M$
 is an isomorphism. 
\end{proposition}

\paragraph{Interchange of colimits.} %
Let $\calC$, $\calD$, and $\mathcal{E}$ be categories. For the product category $\calC\times \calD$, let $F:\calC\times \calD \to \mathcal{E}$ be a functor.
Suppose for all $c\in \ob(\calC)$, the colimit of $F(c,-):\calD\to \mathcal{E}$ exists. Then we have a functor 
$\colim_d F(-,d):\calC \to \mathcal{E}$
, which maps $c\mapsto \colim\, F(c,-)$ and for each morphism $f:c\to c'$ in $\calC$ the natural map induced by the natural transformation $F(f,-)$ \cite[V.3 Theorem 1]{mac2013categories}.
Symmetrically, suppose for all $d\in \ob(\calD)$, the colimit of $F(-,d):\calC\to \mathcal{E}$ exists. Then we have a functor 
$\colim_c F(c,-):\calD\to \mathcal{E}$.
As a direct implication of \cite[IX.2 eq.(2)]{mac2013categories}, we have:

\begin{proposition}\label{lem:colimt_commutes_with_colimit}
    Let $\calC$ and $\calD$ be categories with finite sets of objects and morphisms. Consider a functor $F:\calC\times \calD \to \vect$. For the iterated colimits $\colim_c\colim_d F(c,d)$ and $\colim_d \colim_c F(c,d)$, one of them exists if and only if so does the other. In that case, there is a canonical isomorphism:
    \[ \colim_c \colim_d F(c,d) \cong \colim_d \colim_c F(c,d).\]
\end{proposition}

\subsection{Interleaving distance and presentations for multi-parameter persistence modules}

In this section, we review the notion of interleaving distance between $\R^d$-persistence modules. We also review a module-theoretic viewpoint on $\R^d$-persistence modules, as an alternative to the functorial viewpoint given in the preceding section. Finally, we review the relationship between interleaving and presentations. All of the results in this section can be found in \cite{carlsson2009theory,lesnick2015theory}. See also \cite{miller2020homological}.

For any $\ep\in \R$ and $d\in \N$, let $\vec{\ep}:=(\ep,\ep,\ldots,\ep)\in \R^d$. 

\begin{definition}[{\cite{chazal2009proximity,lesnick2015theory}}]\label{def:interleaving}
    Let $M$ be an $\R^d$-persistence module. Let $M(\ep)$ be the \textbf{$\ep$-translation} of $M$ which is an $\R^d$-persistence module, given by $M(\ep)_a=M_{a+\vec{\ep}}$ for $a\in \R^d$, and $M(\ep)_{a\le b} = M_{a+\vec{\ep}\le b+\vec{\ep}}$ for $a\leq b\in \R^d$.  Let $N$ be another $\R^d$-persistence module.
    We say that $M$ and $N$ are \textbf{$\ep$-interleaved} if there are natural transformations  $F:M\Rightarrow N(\ep)$ and $G:N\Rightarrow M(\ep)$ such that  
    \begin{equation}\label{eq:interleaving}
       \mbox{for all $a\in \R^d$},\ \   \text{{$G_{a+\vec{\ep}} \circ F_a$}} =M_{a\le a+2\vec{\ep}}\ \mbox{ and }\  \text{{$F_{a+\vec{\ep}}\circ G_a$}} =N_{a\le a+2\vec{\ep}}.
    \end{equation} The \textbf{interleaving distance} between $M$ and $N$ is given by
    $$ d_I(M,N)=\inf \{\ep\geq0 : \text{$M$ and $N$ are $\ep$-interleaved}\}.$$
    If $M$ and $N$ are not $\ep$-interleaved for all $\ep\geq0$, then we let $\dint(M,N)=\infty$.
\end{definition}

For any integer $d\ge 1$, let $P_d$ denote the collection of all polynomials in
$d$ indeterminates $x_1,\ldots,x_d$ with nonnegative real exponents.
In other words, $P_d$ consists of all finite formal linear combinations
\(
\sum c_{a_1,\ldots,a_d}\, x_1^{a_1} x_2^{a_2}\cdots x_d^{a_d},
\)
where $(a_1,\ldots,a_d)\in [0,\infty)^d$ and $c_{a_1,\ldots,a_d}\in \F$.
More briefly, we write such an element as
\(
\sum c_a x^a, \ \ a\in[0,\infty)^d.
\)
The space $P_d$ is an $\F$-vector space equipped with the naturally defined
addition and multiplication. From this, one can see that $P_d$ is an $\F$-\emph{algebra}.
The $\F$-algebra $P_d$ is naturally graded by $[0,\infty)^d$ where at grade $a\in [0,\infty)^d$, the vector space $(P_d)_a = \{cx^a: c\in \F\}$ lies. %

\begin{definition}
    We define a \textbf{$d$-graded module} as a $P_d$-module $M$ in the usual algebraic sense such that it has a direct decomposition $M=\bigoplus_{a\in\R^d} M_a$ and 
    satisfies  $x^a(M_b)\subseteq M_{a+b}$. A nonzero element $m\in M$  is said to be \textbf{homogeneous} if $m\in M_a$ for some $a\in \R^d$.
\end{definition}

By virtue of the following remark, $d$-graded modules will often be identified with $\R^d$-persistence modules, and vice versa.
\begin{remark}[{\cite[Remark 2.1]{lesnick2015theory}}]  \label{1}
    For any $d$-graded module $M$, the homogeneous part $M_a$ for any $a\in\R^d$ gives an $\F$-vector space, and for $a\le b\in \R^d$, $x^{b-a}$ defines a linear map from $M_a$ to $M_b$. This defines an $\R^d$-persistence module. Conversely, any $\R^d$-persistence module defines a $d$-graded module in a natural way.
\end{remark}

By this remark, the $\ep$-translation of any $d$-graded module is also well-defined.
Namely, given any $d$-graded module $M$, the $\ep$-translated module $M(\ep)$ is given by $M(\ep)_a=M_{a+\vec\ep}$ for all $a\in \R^d$, and action by $x^b$ for $b\in [0,\infty)^d$ is defined via the structure maps $M(\ep)_a = M_{a+\vec\ep}\to M_{a+b+\vec\ep}= M(\ep)_{a+b}$ for all $a\in \R^d$.

We regard the algebra $P_d$ as a $d$-graded module by defining, for each $a \in [0,\infty)^d$, $(P_d)_a$ to be the one-dimensional $\F$-vector space generated by the monomial $x^a$. %
Similarly, given a singleton set $\{w\}$, let $P_d[w]$ be the free $d$-graded module with basis $\{w\}$, which is isomorphic to the $d$-graded module $P_d$ via the isomorphism $f w\mapsto f$ for all $f\in P_d$. In light of \Cref{1}, both $P_d$ and $P_d[w]$ can be seen as the interval module $\F_{\vec{0}\binuparrow}:\R^d\rightarrow \vect$, as defined in \eqref{eq:interval module}. 

A \textbf{graded set} is a set of points where each point is assigned an element of $\R^d$.
Formally, fix a set of symbols $X$ of arbitrary size. 
A graded set is %
a set of ordered pairs 
$ W\subset X\times \R^d$
of a symbol and a grade. 
We say that a graded set $W$ is \textbf{\faithful} if for each $w\in X$ there is at most one $a\in \R^d$ such that $(w,a)\in W$. %

The \textbf{set of symbols} in $W$, %
\{$w\in X$:\ \mbox{there exists $a\in \R^d$ such that} $(w,a)\in W$\}, is denoted by $\sym(W)$. The set of grades of the points in $W$, \{$a\in \R^d$:\ \mbox{there exists $w\in X$ such that} $(w,a)\in W$\}, is called the \textbf{grade set} of $W$ and denoted by $\gr(W)$. For a faithful graded set $W$, there is a \textbf{grading function} $\gr_W: \sym(W)\to \gr(W)$ %
that maps each symbol its (well-defined) \textbf{grade}. %
For example, any set of homogeneous elements in an $\R^d$-persistence module $M$ is a graded set provided that the associated grades of each point is inherited from $M$.

Suppose $P,Q$ be posets in $\R^d$ and $f:P\to Q$ be monotone. Let $W$ be a graded set whose grade set is contained in $P$. Then we define $f(W)$ to be the graded set $\{(w,f(a)):(w,a)\in W\}$.
Also, for any $\ep\in \R$, we define the \textbf{translated graded set} $W(\ep)= \{ (w,a+\vec\ep):(w,a)\in W\}.$ %

\begin{definition}\label{def:free}
    \begin{enumerate}[label=(\roman*)]
        \item Let $S$ be a set. The \textbf{free $\F$-vector space on basis $S$}, denoted by $\fr_\F[S]$,
        is the space of formal linear combinations $\sum_{s\in S} a_s s$ with $a_s \in \F$, where $a_s = 0$ for all but finitely many $s \in S$. %
        Any set map $f:S_1\to S_2$ induces the linear map $f^*:\fr_\F[S_1]\to \fr_\F[S_2]$ sending $\sum_{s\in S_1} a_ss$ to $\sum_{s\in S_1} a_sf(s)$. 
        \label{item:free1}
        \item  \label{item:free2}
    Let $W$ be a faithful graded set. The 
    \textbf{free $\R^d$-persistence module on basis $W$}
    is the $d$-graded module (see \Cref{fig:free})
    \[fr[W]=\bigoplus_{(w,a)\in W} P_d[w](-a)=\left\{\sum_{w\in W} f_w w: \mbox{$f_w\in P_d$ are zero for all but finitely many $w\in W$}\right\}.\]
    \end{enumerate}
\end{definition}

In what follows, whenever we consider $\fr[W]$ for a graded set $W$, we assume that $W$ is faithful. %

\begin{figure}
\begin{center}
\begin{tikzpicture}[scale=1.0]

  \draw[->] (-0.2,0) -- (4.2,0) node[below] {$x$};
  \draw[->] (0,-0.2) -- (0,4.2) node[left] {$y$};

  \fill[blue!20, opacity=0.6]
    (0,1) -- (4,1) -- (4,4) -- (0,4) -- cycle;
  \draw[blue!70]
    (0,1) -- (4,1) -- (4,4) -- (0,4) -- cycle;
  \fill[blue!70] (0,1) circle (1.5pt);
  \node[blue!70, left] at (0,1) {$w_2$};
  \fill[red!20, opacity=0.6]
    (1,0) -- (4,0) -- (4,4) -- (1,4) -- cycle;
  \draw[red!70]
    (1,0) -- (4,0) -- (4,4) -- (1,4) -- cycle;
  \fill[red!70] (1,0) circle (1.5pt);
  \node[red!70, below] at (1,0) {$w_1$};

  \node[blue] at (-1.6,2.5) {$P_d[w_2](-(0,1))$};
  \node[darkred]  at (5.6,0.5) {$P_d[w_1](-(1,0))$};

\end{tikzpicture}
\end{center}
\caption{Illustration of $\fr[W]\cong \F_{(1,0)\binuparrow}\oplus \F_{(0,1)\binuparrow}$ for the graded set $W:=\{(w_1,(1,0)),(w_2,(0,1))\}.$ %
For example, the $\F$-vector space $\fr[W]_{(2,2)}$ is freely generated by the homogeneous elements
$(x_1x_2^2)w_1$ and $(x_1^2x_2)w_2$ of $\fr[W]$.}\label{fig:free}
\end{figure}

For any two graded sets $W_1$ and $W_2$, we will use the notation $\fr[W_1,W_2]$ for $\fr[W_1\sqcup W_2]$. 

\begin{remark}\label{rem:translationsubmodule} Fix $\ep\geq 0$.
    \begin{enumerate}[label=(\roman*)]
        \item  
        For any $\R^d$-persistence module $M$, 
        we have the bijection $M\to M(\ep)$, which is the disjoint union of the isomorphisms $M_{a+\vec{\ep}} \cong M(\ep)_a$ for all $a\in \R^d$.
        
        For a subset $Y\subset M$ of homogeneous elements, let $Y(\ep)$ denote the image of $Y$ under the canonical bijection $M\to M(\ep)$, %
        and thus $Y(\ep)$ is a set of homogeneous elements in $M(\ep)$. \label{rem:translationsubmodule1}
        \item For a graded set $W$, we have a canonical inclusion $\fr[W(-\ep)]\to \fr[W]$, the map  given by the structure map $\fr[W(-\ep)]_a\cong \fr[W]_{a-\vec{\ep}}\to \fr[W]_a$. Hence, we regard $\fr[W(-\ep)]$ as a submodule of $\fr[W]$. Similarly, we regard $\fr[W_1(-\ep),W_2]$ as a submodule of $\fr[W_1,W_2]$ \cite[Remark 4.2]{lesnick2015theory} \label{rem:translationsubmodule2}.
        \item Let $Y$ be a homogeneous subset of $\fr[W_1, W_2]$. Since $Y(\ep)$ can be regarded as a homogeneous subset of $\fr[W_1(\ep), W_2(\ep)]$, for any $\delta, \eta\ge \ep$, the set $Y(\ep)$ can be regarded as a homogeneous subset of $\fr[W_1(\delta), W_2(\eta)]$.\label{rem:translationsubmodule3}
    \end{enumerate}
\end{remark}

\begin{definition}\label{def:presentation}
\begin{enumerate}[label=(\roman*)]
    \item Let $S$ be a set %
     and let $R\subset \fr_\F[S]$. Let $\langle R \rangle$ denote the subspace of $\fr_\F[S]$ generated by $R$. We call the quotient space $\langle S|R\rangle_\F:=\fr_\F[S]/\langle R\rangle$ the \textbf{presentation of an $\F$-vector space} with set of generators $S$ and set of relations $R$.
    \label{item:presentation of a vector space}
    \item Let $W$ be a graded set. A \textbf{relation} in $\fr[W]$ is a pair $(r,a)$ such that $r$ is an $\F$-linear combination of symbols in $W$, and $a\in \R^d$ is greater than the grade of each symbol in $r$ with nonzero coefficient. 
    For such an $r=\sum_{i}c_iw_i$, i.e., $c_i\in \F\setminus\{0\}$ and $\gr(w_i)\le a$, we associate a homogeneous element $\sum_{i} c_i x^{a-\gr(w_i)} w_i$ of grade $a$ in $\fr[W]$ which we will also denote by $r$. \label{item:relation_in_free_module}

    \item Let $W$ be a graded set, and $Y$ be a set of relations in $\fr[W]$. We can also regard $Y$ as a homogeneous subset of $\fr[W]$, in virtue of \Cref{item:relation_in_free_module}.
    Let $\langle Y \rangle$ be the $d$-graded submodule of $\fr[W]$ generated by $Y$. We call $\left< W|Y\right>:=\fr[W]/\langle Y\rangle$ the \textbf{presentation of an $\R^d$-persistence module} %
    with generator set $W$ and relation set $Y$. We call 
    $\langle W|Y\rangle $ \textbf{finite} if both $W$ and $Y$ are finite. %
    \label{item:presentation of a module}
   \item For each $i=1,2$, let $W_i$ be a graded set and $Y_i$ be another graded set that is a set of relations in $\fr[W_i]$. We write $\langle W_1, W_2|Y_1, Y_2\rangle$ for $\langle W_1\sqcup W_2 | Y_1\cup Y_2\rangle$. \label{item:union}
\end{enumerate}   
\end{definition}

Let $M\cong\langle W|Y\rangle$ be a presentation of an $\R^d$-persistence module. Although $Y$ may not be faithful, we regard the set of symbols in $Y$ as the disjoint union $\sym(Y)=\bigsqcup_{y\in Y} \pi_1(y)$, and define the \textbf{grading function} of $M$ (w.r.t. the given presentation) $\gr_M:\sym(W)\sqcup \sym(Y)\to \R^d$ as the disjoint union of the grading functions $\gr_W$ and $\gr_Y$. %

\begin{definition}[Alternative of \Cref{KanExtensionOfPersistenceModules}~\ref{item:constructible}]\label{def:alternative definition of finitely presented over P} 
Let $P\subset\R^d$ be a finite subset. An $\R^d$-persistence module $M$ is \textbf{finitely presented over $P$} (or simply $P$\textbf{-presented}) if $M$ admits a finite presentation $\langle W| Y\rangle$ such that $\gr(W), \gr(Y)\subset P$.  We call $M$ \textbf{finitely presented} if $M$ is $P$-presented for some finite $P\subset \R^d$. 
\end{definition}

The equivalence between the preceding definition of finitely presentedness and the one in \Cref{KanExtensionOfPersistenceModules}~\ref{item:constructible} follows from \cite[Proposition 10.3 (2)]{chacholski2021realisations}. The equivalence between the preceding definition of $P$-presentedness and that in \Cref{KanExtensionOfPersistenceModules}~\ref{item:constructible} is not stated in \cite[Proposition 10.3 (2)]{chacholski2021realisations}, but can be inferred from its proof (and the proofs of 
\cite[Lemma 2.1]{bauer2025multi}, \cite[Remark 4.12]{aoki2025preservation}, and \cite[Proposition 2.9]{lesnick2015interactive}).

\begin{remark}
    Let $f:P\to Q$ a monotone function between posets. Let $W$ and $Y$ be graded sets such that $Y$ is a set of homogeneous elements in $\fr[W]$. Since $f$ is monotone, the elements of $f(Y)$ are homogeneous elements of $\fr[f(W)]$. Thus, if $\langle W|Y\rangle$ is a presentation of an $\R^d$-persistence module, then so is $\langle f(W)|f(Y)\rangle$. 
\end{remark}

\begin{remark}\label{rem:constructible_over_grid_closure} Let $M$ be an $\R^d$-persistence module with a finite presentation $\langle W \vert Y\rangle$. Then $M$ is $P$-presented where $P$ is any poset containing $\gr_{W\sqcup Y}(W\sqcup Y)$.
In particular, $P$ can be taken to be the smallest grid containing $\gr_{W\sqcup Y} (W\sqcup Y)$ (cf. \Cref{rem:smallest grid}). 
\end{remark}

\begin{example} \label{ex:presentations_are_valid}
    Let $\ep>0$ and let there be graded sets $W_1,W_2$ and homogeneous sets $Y_1\subset \mathrm{fr}[W_1,W_2(-\ep)]$, $Y_2\subset \mathrm{fr}([W_1(-\ep),W_2])$. Then, by \Cref{rem:translationsubmodule}~\ref{rem:translationsubmodule3},   we have
\[Y_1(-\ep)\subset \fr[W_1(-\ep),W_2]\ \mbox{and}\ Y_2(-\ep)\subset \fr[W_1,W_2(-\ep)].\]
       According to \Cref{def:presentation}~\ref{item:presentation of a module} and~\ref{item:union}, the following are valid presentations:
         \[M:=\langle W_1,W_2(-\ep)|Y_1,Y_2(-\ep)\rangle \ \mbox{and}\ N:=\langle W_1(-\ep),W_2|Y_1(-\ep),Y_2\rangle.\]
\end{example}

The following theorem establishes a connection between the interleaving distance and the  presentations of $\R^d$-persistence modules.

\begin{theorem}[{\cite[Theorem 4.4]{lesnick2015theory}}]\label{lesnickpresentation}
 Let $M$ and $N$ be $\R^d$-persistence modules. Then $M$ and $N$ are $\ep$-interleaved if and only if there exist $d$-graded sets $W_1, W_2$ and homogeneous subsets $Y_1\subset \fr[W_1,W_2(-\ep)]$ and $Y_2 \subset \fr[W_1(-\ep),W_2]$ such that 
    \begin{equation}\label{eq:lesnick_presentation}
        M \cong \langle W_1,W_2(-\ep)|Y_1,Y_2(-\ep)\rangle,\  \ \ \ \  \ \ 
        N \cong \langle W_1(-\ep),W_2|Y_1(-\ep),Y_2\rangle.
    \end{equation}
    Also, when $M$ and $N$ are finitely presented, we can take the presentations
    above to be finite.
    \end{theorem}

\section{Galois-edit distance on $\R^d$-persistence modules}
\label{sec:Edit distance for persistence modules}
In this section, we introduce three edit distances on finitely presented $\R^d$-persistence modules.
Later, we will show that all of these distances coincide with the interleaving distance. 

We begin by considering some special finite subposets in $\R^d$ that serve as subposets determining $\R^d$-persistence modules. All such  subposets will be assumed to be \emph{full} subposets of $\R^d$.

Recall that, for any integer $d\ge 1$, a $d$-grid (or simply a grid) is a poset of the form
$X_1 \times \cdots \times X_d$, where each $X_i \subset \R$ is finite, with
the order defined by
\(
  (x_1,\dots,x_d) \le (y_1,\dots,y_d)
  \quad\text{iff}\quad
  x_i \le y_i \text{ in } \R \text{ for all } i.
\)
We define a $(d$-)\textbf{grid morphism} as a monotone function between $d$-grids of the
form $\prod_{i=1}^d f_i$, where each $f_i$ is a monotone function between
the corresponding $i$-th coordinate sets.

\begin{remark} \label{rem:productofadjoints3}
 The left or right adjoint of any grid morphism, whenever it exists, is again a grid morphism, as we now show. Let a $d$-grid morphism $f:=\prod_{i\in [d]} f_i$ be a left (resp., right) adjoint of a Galois connection. By \Cref{productofadjoints0}, $f$ preserves the bottom (resp., top).
        This implies that for each $i\in [d]$, $f_i$ preserves the bottom (resp., top), thus $f_i$ has the right (resp., left) adjoint $g_i$. It is easy to see that the grid morphism $\prod_{i\in [d]} g_i$ is the right (resp., left) adjoint of $f$. 
\end{remark}

We call a finite subposet $P\subset \R^d$ a \textbf{\joingrid{}} if $P$ is a join-semilattice and the inclusion $P\hookrightarrow \R^d$ preserves all binary joins. Equivalently, a finite subposet $P\subset \R^d$ is a \joingrid{} if it is closed under binary joins in $\R^d$. 
Let $P$ be a \joingrid{}. By \Cref{productofadjoints}~\ref{item:productofadjoints2}, there is a floor function $\lfloor-\rfloor^{\R^d}_P:\R^d_\bot\to P_\bot$, which we simply write $\lfloor - \rfloor_P$.

\begin{remark}
\label{item:Kan_extension_by_floor_function1} When $P\subset \R^d$ is a join-semigrid, we have the following characterization of $P$-presented $\R^d$-persistence modules: An $\R^d$-persistence module $M$ is $P$-presented if and only if $M_\bot = M_\bot \circ \lfloor - \rfloor_P$. This statement is proved as follows.

    First, note that for any two $\R^d_\bot$-persistence modules $M$ and $N$ such that $M(\bot)=N(\bot)=0$, $M= N$ if and only if $M|_{\R^d}= N|_{\R^d}$. In particular, $M_\bot=M_\bot \circ \lfloor-\rfloor_P$ if and only if $M=M_\bot \circ \lfloor-\rfloor_P|_{\R^d}$. Next, the left Kan extension of any functor $M:P\to \vect$ along the inclusion $\iota:P\hookrightarrow \R^d$, denoted by $\Lan_\iota M$, is (up to isomorphism) the $\R^d$-persistence module such that
    for $a\in \R^d$ 
    $$\left( \Lan_\iota M\right)_a=\begin{cases}
        M(\lfloor a\rfloor_P), & \text{if } \lfloor a \rfloor_P \in P\\
        0, & \text{otherwise}
    \end{cases}$$
    and for $a\le a'$, the map $\left(\Lan_\iota M\right)_a\to\left(\Lan_\iota M\right)_{a'}$ is naturally defined \cite[Theorem 6.2.1]{riehl2017category}.\footnote{ 
    In particular, when $P$ is a $d$-grid, the left Kan extension along the inclusion $P\hookrightarrow \R^d$ is referred to as an \emph{extension functor} \cite[Definition 7.9]{blanchette2023exact}.}
    Therefore, we have $\Lan_\iota M=M_\bot \circ \lfloor-\rfloor_P|_{\R^d}$. %
    By \Cref{KanExtensionOfPersistenceModules}~\ref{item:constructible}, it follows that $M$ is $P$-presented if and only if $M = M_\bot \circ \lfloor - \rfloor_P|_{\R^d}$ (or equivalently, $M_\bot = M_\bot \circ \lfloor-\rfloor_P$).\footnote{In the setting of \Cref{item:Kan_extension_by_floor_function1}, \cite{chacholski2021realisations} refers to the floor function $\lfloor - \rfloor_P$ as a \emph{transfer}, to $M$ as being \emph{discretised} by $\iota$, and to $M|_P$ as a \emph{discretisation} of $M$. %
    }
\end{remark}

\begin{definition} \label{categoryofposets} For any $d\in \N$, we define three categories $\Poset$, $\jsLat$ and  $\carte$
such that a latter is a subcategory of a former.

\label{def:differentclassesforpersistencemoduledomain}
    \begin{enumerate}[label=(\roman*)]  
        \item $\Poset$ is the category consisting of the following:
        \begin{itemize}
            \item Objects: finite posets in $\R^d$,
            \item Morphisms: monotone functions for which the right adjoint exists. 
        \end{itemize}
        \item $\jsLat$ is the category consisting of the following:
        \begin{itemize}
            \item Objects: \joingrid{}s in $\R^d$ %
             (note: any \joingrid{} has the top element),
            \item Morphisms: monotone functions for which the right adjoint exists.
        \end{itemize}
        \item $\carte$ is the category consisting of the following:
        \begin{itemize}
            \item Objects: $d$-grids,
            \item Morphisms: grid morphisms for which the right adjoint exists (note: the right adjoint is also a grid morphism by \Cref{rem:productofadjoints3}).
        \end{itemize} \label{def:1Dmorphism}
    \end{enumerate}
\end{definition}

\Cref{remark:galois basics}~\ref{rem:compositionofGalois} guarantees that each of $\Poset$, $\jsLat$, and $\carte$ forms a category. In particular, when $d=1$, the categories $\Poset$, $\jsLat$, and $\carte$ coincide. %
However, when $d \ge 2$, $\Poset$ contains strictly more morphisms than $\jsLat$, which in turn contains strictly more morphisms than $\grid$; namely there exists a morphism in $\Poset$ (resp. $\jsLat$) that is not an isomorphism and does not belong to $\jsLat$ (resp. $\grid$). For instance, let $d=2$ and $$P:=\{(1,0),(0.5,0),(0,1),(0,0.5)\}, \quad \mbox{and} \quad Q:=\{(1,0),(0,1)\}.$$ Define the monotone function $f:P\to Q$ by
    $
    f(p_1,p_2)=(\lceil p_1\rceil_\Z,\lceil p_2\rceil_\Z)
    $
    where $\lceil -\rceil_\Z$ is the standard ceiling function $\R\to \Z$.
    Then $f$ admits the inclusion $\iota:Q\hookrightarrow P$ as the right adjoint. %
    Thus, $f$ is a morphism in $\Poset$, which is not an isomorphism. Also, since $P$ and $Q$ are not join-semilattices, $f$ is not a morphism in $\jsLat$. 
    
 We also remark that, since grids are \emph{finite} lattices, grid morphisms preserve all binary joins if and only if they preserve all nonempty joins. In fact, grid morphisms preserve all binary joins, and thus a grid morphism preserves all joins if and only if it preserves the bottom (i.e., the empty join). This implies that, by \Cref{productofadjoints}~\ref{item:productofadjoints1}, morphisms in the category $\carte$ are exactly grid morphisms that preserve bottoms.

 In what follows, we let $\calE$ be one of $\Poset$, $\jsLat$, or $\carte$.  Given any category $\calC$, let $\ob(\calC)$ denote the collection of objects in $\calC$.  
Let $M$ and $N$ be finitely presented $\R^d$-persistence modules. Suppose that $M$ is $P$-presented and $N$ is $Q$-presented for $P,Q \in \mathrm{ob}(\calE)$.
   An \textbf{$\calE$-edit} from $M$ to $N$ is a morphism $f:P\rightarrow Q$ in $\calE$ such that  $M|_P\circ g\cong N|_Q$, where $g:Q\rightarrow P$ is the right adjoint of $f$. In other words, $f$ is an $\calE$-edit from $M$ to $N$ iff $N|_Q$ is naturally isomorphic to the left Kan extension of $M|_P$ along $f$.

\begin{definition}\label{def:edit distance} 
Let $f:P\to Q$ be a monotone function between finite subsets $P,Q\subseteq\R^d$.
    The \textbf{maximum displacement} of $f$ is $\dist{f}:=\max \{\norm{p-f(p)}_{\infty}:p\in P \}.$
An \textbf{$\calE$-edit sequence} between finitely presented $\R^d$-persistence modules $M$ and $N$ is a finite sequence $M=M_0, \ldots, M_m=N$ of $\R^d$-persistence modules, together with a sequence of $\calE$-\edit{}s $f_1,\dots, f_m$ where each $f_i$ is an $\calE$-\edit{} $M_{i-1}\rightarrow M_i$ or $M_{i}\rightarrow M_{i-1}$. The \textbf{cost} of the $\calE$-edit sequence is the sum $\sum_{i=1}^m \lVert f_i \rVert$. 
The \textbf{$\calE$-edit distance} between $M$ and $N$
is defined as the infimum of the cost of all $\calE$-edit sequences between $M$ and $N$, denoted by $\dedi^\calE(M,N)$. If there is no such edit sequence, then $\dedi^\calE(M,N)$ is defined to be $\infty$. 

We also refer to an $\calE$-edit sequence and the distance $\dedi^\calE$ as a \textbf{Galois-edit sequence} and the 
\textbf{Galois-edit distance}, respectively.
\end{definition}

 It is easy to see that $\dedi^\calE$ is an extended pseudometric on the collection of finitely presented $\R^d$-persistence modules. In what follows, any extended pseudometric will simply be referred to as a metric. 

\begin{remark}\label{relationbetweenedits} 
\begin{enumerate}[label=(\roman*)] 
\item  The inclusions of the categories imply that
\(\dedi^{\Poset}\le \dedi^{\jsLat}\le \dedi^{\carte}.\label{eq:1}\)
\label{item:relationbetweenedits1}
\item 
More generally, let $\calC$ be any category 
that contains $\carte$ and is contained in $\Poset$. Then, we obtain the corresponding edit distance $\dedi^\calC$, defined in the same way as in the preceding definition, and we have
\(\dedi^\Poset\le \dedi^\calC \le \dedi^\carte.\)
\end{enumerate}
\end{remark}

 Next, we will investigate basic properties of $\dedi^\calE$.

\begin{proposition}\label{prop:distance zero}
If there exists an $\calE$-edit sequence of cost zero between two finitely presented $\R^d$-persistence modules, then the two modules are naturally isomorphic as functors $\R^d\to\vect$. 
\end{proposition}

The following lemma is useful for proving the preceding proposition. 

\begin{lemma} \label{prop:trivial_edit}
Let an inclusion $f:P \hookrightarrow Q$ be a morphism in $\calE$, and let $M$ and $N$ be $P$- and $Q$-presented $\R^d$-persistence modules, respectively. 
Then $f$ is an $\calE$-\edit{} $M\to N$ if and only if $M \cong N$.
\end{lemma}

\begin{proof}%
Let $\iota_P:P\hookrightarrow \R^d$, $\iota_Q:Q\hookrightarrow \R^d$ be the inclusions. %
We recall that $f$ is an $\calE$-edit $M\to N$ if and only if $N|_Q\cong \mathrm{Lan}_f M|_P$. 
We first prove the forward implication. Suppose $N|_Q\cong \mathrm{Lan}_f M|_P$. This assumption and the functoriality of left Kan extension guarantee the isomorphisms ($\ast\ast$) and ($\ast$) below, respectively.
\[M\cong \Lan_{\iota_P} M|_P \stackrel{(\ast)}{\cong} \Lan_{\iota_Q} \left(\Lan_{f} M|_P\right) \stackrel{(\ast\ast)}{\cong} \Lan_{\iota_Q} N|_Q \cong N.
\]
Conversely, suppose $M\cong N$. Then, we have the isomorphism $(\ast\ast)$, because of the other isomorphisms above and the assumption $M\cong N$. 
Now, by restricting both $\Lan_{\iota_Q} \left(\Lan_{f} M|_P\right)$ and $\Lan_{\iota_Q} N|_Q$ to $Q$, we obtain the desired isomorphism $\Lan_f M|_P \cong N|_Q$.
 \end{proof}

\begin{proof}[Proof of \Cref{prop:distance zero}]
Consider any $\calE$-edit sequence $M=M_0\stackrel{f_1}{\leftrightarrow} \cdots\stackrel{f_{n}}{\leftrightarrow}M_n=N$ whose cost is zero. Then, for each $i$, $\dist{f_i}=0$ and thus $f_i$ is a poset inclusion. \Cref{prop:trivial_edit} implies that $M_{i-1}\cong M_{i}$ for each $i$, and thus $M\cong N$.
\end{proof}

Next, we provide a characterization of a pair of persistence modules that have finite edit distance.
Let $P\in \ob(\calE)$ and let $M$ be a $P$-presented $\R^d$-persistence module. %
 Since $P$ is a finite join-semilattice,  $\bigvee P=\top_P\in P$. In addition,
 for any $a \ge \top_P$, 
we have that $M_{\top_P}\cong M_a$, and
 $M_{\top_P\le a}$ is the identity map. Note that the space $M_{\top_P}$ does not depend on the choice of $P\in \ob(\calE)$ as long as $M$ is $P$-presented.  Thus,  $M(\infty):=M_{\top_P}$ is well-defined.

\begin{proposition}\label{prop:components_of_category}
    Given any finitely presented $\R^d$-persistence modules $M$ and $N$, $\dedi^{\calE}(M,N)<\infty$ if and only if $M(\infty)\cong N(\infty)$ (proved in \Cref{appendix:omitted proofs}).
\end{proposition}

\begin{figure}
    \centering
    \begin{tikzpicture}[scale=1.2, every node/.style={scale=0.9}]
  \begin{scope}[scale= 1, shift={(0,0)}]%
    
    \fill[\persistcolor!40] (0,0) rectangle (1,1);
    \fill[\persistcolor!40] (0,1) rectangle (1,2.3);
    \fill[\persistcolor!60] (1,0) rectangle (2,1);
    \fill[\persistcolor!40] (2,0) rectangle (3.3,1);
    \fill[\persistcolor!40] (0,1) rectangle (1,2.3); 
    \fill[\persistcolor!40] (1,1) rectangle (2,2.3);
    \fill[\persistcolor!20] (2,1) rectangle (3.3,2.3);

    \node (b) at (1.5,0.5) {$\F^2$};
    \node (c) at (2.5,0.5) {$\F$};
    \node (d) at (0.5,1.5) {$\F$};
    \node (f) at (2.5,1.5) {$0$};

    \draw[->, thick] (0,-0.5) -- (0,2.5) node[above] {$y$};
    \draw[->, thick] (-0.7,0) -- (3.5,0) node[right] {$x$};

    \draw[-, thin, color=\persistcolor!40] (0,1) -- (1,1);
    \draw[-, thin, color=\persistcolor!40] (1,1) -- (1,2.3);
    \draw[-, thin] (1,1) -- (3.3,1);
    \draw[-, thin] (1,0) -- (1,1);
    \draw[-, thin] (2,0) -- (2,2.3);

    \foreach \x in {0,1,2} 
    \foreach \y in {0,1}
    \draw[color=black!80,fill=orange!40, thick] (\x,\y) circle (\pointsize) ;

    \node at (-0.5,1) {\textcolor{orange!80}{$P$}};
  \end{scope}

  \begin{scope}[shift={(5,0)}]

    \fill[red!40] (0,0) rectangle (2,2.3); %
    \fill[red!40] (0,0) rectangle (3.3,1); %
    \fill[red!20] (2,1) rectangle (3.3,2.3); %

    \node at (1,0.8) {$\F$};
    \node at (2.5,1.5) {$0$};

    \draw[->, thick] (0,-0.5) -- (0,2.5) node[above] {$y$};
    \draw[->, thick] (-0.7,0) -- (3.5,0) node[right] {$x$};

    \draw[-, thin] (2,1) -- (3.3,1);
    \draw[-, thin] (2,1) -- (2,2.3);
    
    \foreach \x in {0,2} 
    \foreach \y in {0,1}
    \draw[color=black!80,fill=blue!50] (\x,\y) circle (\pointsize) ;

    \node at (-0.5,1) {\textcolor{blue!70}{$Q$}};
  \end{scope}

  \node at (1.5,-0.7) {$M$};
  \node at (6.5,-0.7) {$N$};
\end{tikzpicture}

    \caption{Illustration of $M$ and $N$ from Example~\ref{ex:edit_makes_simpler}. The grids $P$ and $Q$ are depicted as 6 dots and 4 dots, respectively.
    }
    \label{fig:example_of_an_edit} 
\end{figure}

Recall that $\calE$-\edit{}s are left adjoints between posets (i.e., the left maps in Galois connections). 
The following proposition implies that any $\calE$-edit sequence may be taken to consist solely of surjective monotone functions whose right adjoints are order embeddings.

\begin{proposition}\label{rem:gisanembedding}
Let $f:M\to N$ be an $\calE$-\edit{}, i.e., $f$ is a morphism $P\to Q$ in $\calE$ (with additional conditions). Consider the restriction $f' : P \to f(P)$ of the codomain of $f$. 
Then $f'$ is also an $\calE$-\edit{} $M\to N$, and its right adjoint is an order embedding (note: $\dist{f}=\dist{f'}$). 
\end{proposition}    

This proposition shows that every surjective edit corresponds to a generalized notion of a \emph{closure system} \cite[Definition 2]{erne1993primer}, as follows. A closure system of a poset $P$ is defined as a subset $Q$ of $P$ such that for each $p\in P$ there exists a smallest $q\in Q$ with $p\leq q$, i.e. $Q$ is the full subposet of $P$ such that the inclusion $Q\hookrightarrow P$ admits a left adjoint. A surjective edit fits naturally into this framework, as it is a left adjoint to an injection (not necessarily an inclusion).

\begin{proof}
   Let $g$ be the right adjoint of $f$, and let $\iota$ be the inclusion $f(P)\hookrightarrow Q$. %
   By \Cref{remark:galois basics}~\ref{rem:decomposition_of_Galois_connection}, we have two Galois connections:\[f':P\leftrightarrows f(P): g|_{f(P)} \ \ \ \ \mbox{ and }  \ \ \  \ \iota:f(P)\leftrightarrows Q:fg\]    
     Before proving that $f'$ is an $\calE$-edit, we show that the codomain $f(P)$ is actually in $\calE$. When $\calE=\Poset$, there is nothing to show. If $\calE=\grid$, $f$ is a grid morphism and therefore $f(P)$ is a grid. If $\calE=\jsLat$, as $f$ preserves all joins (\Cref{productofadjoints}~\ref{productofadjoints0}), $f(P)$ is closed under all binary joins; %
     indeed, let $a,b\in f(P)$, and choose $c\in f^{-1}(a)$ and $d\in f^{-1}(b)$ arbitrarily. Then the join $c\lor d$ exists in $P$, and its image $f(c\lor d)$ coincides with the join of $a$ and $b$ in $Q$. In particular, this join lies in $f(P)$. Therefore, every binary join in $f(P)$ agrees with that in $Q$, which in turn agrees with that in $\R^d$, as $Q$ is a \joingrid. 
     This shows that $f(P)$ is also   a \joingrid. %
    
    By \Cref{prop:trivial_edit},  $\iota$ is an $\calE$-\edit{} $N \to N$,
    which implies that $N$ is $f(P)$-presented. Also, since $M|_P \circ g\cong N|_{Q}$, we have that $M|_P \circ g|_{f(P)}\cong N|_{f(P)}$. Therefore, $f'$  is also an $\calE$-\edit{} $M\to N$ with the right adjoint $g|_{f(P)}$. 
    
    Next, we show that $g|_{f(P)}$ is an order embedding.     
    By \Cref{remark:galois basics}~\ref{remark:triangle_identity}, we have $f(g(f(p)))=f(p)$ for every $p\in P$. Hence $f'\circ g|_{f(P)}=\id_{f(P)}$. For any $q, q'\in f(P)$, we have:
    \[
    q\le q' \implies g|_{f(P)}(q)\le g|_{f(P)}(q') \implies f'(g|_{f(P)}(q))\le f'(g|_{f(P)}(q'))\implies q\le q'.
    \]
    Therefore, the right adjoint $g|_{f(P)}$ is  an order embedding. 
\end{proof}

\begin{example}[A $\mathsf{Grid}_2$-edit] \label{ex:edit_makes_simpler}
    For a 2-grid $P:=\{0,1,2\}\times \{0,1\}$ and its subgrid $Q:=\{0,2\}\times\{0,1\}$, let $M$ and $N$ be $P$- and $Q$-presented $\R^2$-persistence modules, respectively.
    Assuming their restrictions on $P$ and $Q$ respectively are as follows, $M$ and $N$ are depicted as in \Cref{fig:example_of_an_edit}.
    
\adjustbox{scale=1,center}
{\centering$$
\begin{tikzcd}[ampersand replacement=\&]
  {} \arrow[d, phantom, "M|_P:" description] \&
  \F \arrow[r,"1"] \&
  \F \arrow[r] \&
  0 \\
  {} \&
  \F \arrow[u,"1"]
    \arrow[r, "{\scalebox{0.75}{$\begin{pmatrix} 1\\ 0 \end{pmatrix}$}}"'] \&
  \F^2 \arrow[r,"(1\ 1)"] \arrow[u,"(1\ 0)"] \&
  \F \arrow[u]
\end{tikzcd}
\hspace{15mm}
\begin{tikzcd}[ampersand replacement=\&]
  {} \arrow[d, phantom, "N|_Q:" description] \&
  \F \arrow[r] \&
  0  \\
  {} \&
  \F \arrow[u,"1"]
    \arrow[r, "1"'] \&
  \F\arrow[u] 
\end{tikzcd}$$}
    The inclusion $g:Q\hookrightarrow P$ preserves all meets. Therefore, by \Cref{productofadjoints}~\ref{item:productofadjoints1}, $g$ admits a left adjoint $f$. %
    As $M|_P\circ g = N|_Q$, $f$ is a $\mathsf{Grid}_2$-\edit{} from $M$ to $N$. 
    
\end{example}

    The example demonstrates the fact that, given a $\carte$-\edit{} $M\to N$, every vector space and structure map present in $N$ is also present in $M$. This example also shows that there can be a $\carte$-\edit{} from a non-interval-decomposable module to an interval (decomposable) module. Since any $\carte$-\edit{} is also a $\Poset$-\edit{}, these two observations remain valid even for $\Poset$-\edit{}s. 

\begin{figure}
    \centering
\begin{tikzpicture}[scale=0.8]
\definecolor{Rcolor}{RGB}{41,182,43}
    \def\ponecolor{orange}
    \def\pthreecolor{green}
    \def\pfivecolor{magenta}

\begin{scope}[shift={(0,0)}]
\begin{scope}[scale= 2, shift={(0,0)}]%
    
    \fill[\persistcolor!40] (0,0) rectangle (1,1);
    \fill[\persistcolor!40] (0,1) rectangle (1,2.3);
    \fill[\persistcolor!60] (1,0) rectangle (2,1);
    \fill[\persistcolor!40] (2,0) rectangle (3.3,1);
    \fill[\persistcolor!40] (0,1) rectangle (1,2.3); 
    \fill[\persistcolor!40] (1,1) rectangle (2,2.3);
    \fill[\persistcolor!20] (2,1) rectangle (3.3,2.3);

    \node (b) at (1.5,0.5) {$\F^2$};
    \node (c) at (2.5,0.5) {$\F$};
    \node (d) at (0.5,1.5) {$\F$};
    \node (f) at (2.5,1.5) {$0$};

    \draw[->, thick] (0,-0.5) -- (0,2.5) node[above] {$y$};
    \draw[->, thick] (-0.7,0) -- (3.5,0) node[right] {$x$};

    \draw[-, thin, color=\persistcolor!40] (0,1) -- (1,1);
    \draw[-, thin, color=\persistcolor!40] (1,1) -- (1,2.3);
    \draw[-, thin] (1,1) -- (3.3,1);
    \draw[-, thin] (1,0) -- (1,1);
    \draw[-, thin] (2,0) -- (2,2.3);

    \foreach \x in {0,1,2} 
    \foreach \y in {0,1}
    \draw[color=black!80,fill=\ponecolor!40, thick] (\x,\y) circle (\pointsize) ;

    \node at (-0.5,1) {\textcolor{\ponecolor!80}{$P$}};
    \node at (1.5,-0.3) {$M$};
  \end{scope}

  \draw[->, thick] (3,-4) --node[anchor=east] {$f_1$} (3,-1.8);

  \begin{scope}[scale = 2, shift={(0,-5)}]%
    
    \fill[\persistcolor!40] (0,0) rectangle (1,2.3);
    \fill[\persistcolor!60] (1,0) rectangle (2,1);
    \fill[\persistcolor!60] (1,0) rectangle (3.3,0.3);

    \fill[\persistcolor!40] (0,0.3) rectangle (1,1);
    \fill[\persistcolor!40] (2,0.3) rectangle (3.3,0.7);

    \fill[\persistcolor!20] (2,0.7) rectangle (3.3,2.3);
    
    \fill[\persistcolor!40] (0,1) rectangle (2,2.3); 

    \node at (1.5,0.5) {$\F^2$};
    \node at (2.5,0.5) {$\F$};
    \node at (0.5,1.5) {$\F$};
    \node at (2.5,1.5) {$0$};

    \draw[->, thick] (0,-0.5) -- (0,2.5) node[above] {$y$};
    \draw[->, thick] (-0.7,0) -- (3.5,0) node[right] {$x$};
    
    \draw[-, thin] (1,1) -- (2,1);
    \draw[-, thin] (2,0.3) -- (3.3,0.3);
    \draw[-, thin] (2,0.7) -- (3.3,0.7);
    \draw[-, thin] (1,0) -- (1,1);
    \draw[-, thin] (2,0.3) -- (2,2.3);

    \foreach \x in {0,1,2}
    \foreach \y in {0,0.3,0.7,1}
    {
        \pgfmathsetmacro\valueone{ (\x != 2.3) && ((\y == 0.3) || (\y == 1)) ? 1 : 0 }

        \pgfmathsetmacro\valuetwo{ (\x != 2.3) && ((\y != 0.3)) ? 1 : 0 }

        \ifthenelse{\valueone=1 \AND\valuetwo =1}{
            \draw[color=black!80,fill=\ponecolor!40, thick, xshift=0.3mm, yshift=0.3mm] (\x,\y) circle (\pointsize);
            \draw[color=black!80,fill=\pthreecolor!40, thick, xshift=-0.3mm, yshift=-0.3mm] (\x,\y) circle (\pointsize);
        }
    
        \ifthenelse{\valueone=1 \AND \valuetwo =0}{
            \draw[color=black!80,fill=\ponecolor!40, thick] (\x,\y) circle (\pointsize);
        }{}

        \ifthenelse{\valuetwo=1 \AND \valueone = 0}{
            \draw[color=black!80,fill=\pthreecolor!40, thick] (\x,\y) circle (\pointsize);
        }{}

        \pgfmathparse{!\valueone && !\valuetwo}

        \ifthenelse{\pgfmathresult=1}{
            \draw[color=black!80,fill=\pointcolor!40, thick] (\x,\y) circle (\pointsize);
        }{}
    }

    \node at (-0.5,1) {\textcolor{blue!60}{$Q$}};
    
    \node at (1.5,-0.3) {$M_1$};
  \end{scope}

  \draw[->, thick] (3.5,-12) --node[anchor=north east] {$f_2$} (5.5,-14);

  \begin{scope}[scale = 2,shift={(2.5,-10)}]%
    
    \fill[\persistcolor!40] (0,0) rectangle (1,2.3);
    \fill[\persistcolor!60] (1,0) rectangle (3.3,0.5);
    \fill[\persistcolor!60] (1,0) rectangle (2,1);
    \fill[\persistcolor!20] (2,0.5) rectangle (3.3,2.3);
    
    \fill[\persistcolor!40] (0,1) rectangle (1,2.3); 
    \fill[\persistcolor!40] (1,1) rectangle (2,2.3);

    \node (b) at (1.5,0.5) {$\F^2$};
    \node (d) at (0.5,1.5) {$\F$};
    \node (f) at (2.5,1.5) {$0$};

    \draw[->, thick] (0,-0.5) -- (0,2.5) node[above] {$y$};
    \draw[->, thick] (-0.7,0) -- (3.5,0) node[right] {$x$};

    \draw[-, thin, color=\persistcolor!40] (0,1) -- (1,1);
    \draw[-, thin, color=\persistcolor!40] (1,1) -- (1,2.3);
    \draw[-, thin] (1,1) -- (2,1);
    \draw[-, thin] (1,0) -- (1,1);
    \draw[-, thin] (2,0.5) -- (2,2.3);
    \draw[-, thin] (2,0.5) -- (3.3,0.5);

    \foreach \x in {0,1,2} 
    \foreach \y in {0,0.5,1}
    \draw[color=black!80,fill=\pthreecolor!40, thick] (\x,\y) circle (\pointsize) ;

    \node at (-0.5,1) {\textcolor{Rcolor!90}{$R$}};

    \node at (1.5,-0.3) {$M_2$};
  \end{scope}
\end{scope}

\begin{scope}[shift={(10,0)}]

  \draw[->, thick] (3,-4) --node[anchor=west] {$f_4$} (3,-1.7) ;
    
  \begin{scope}[scale = 2, shift={(0,-5)}]
    
    \fill[\persistcolor!40] (0,0) rectangle (1,2.3);
    \fill[\persistcolor!60] (1,0) rectangle (2,1);
    \fill[\persistcolor!60] (1,0) rectangle (3.3,0.3);

    \fill[\persistcolor!40] (0,0.3) rectangle (1,1);
    \fill[\persistcolor!40] (2,0.3) rectangle (3.3,0.7);

    \fill[\persistcolor!20] (2,0.7) rectangle (3.3,2.3);
    
    \fill[\persistcolor!40] (0,1) rectangle (2,2.3); 

    \node at (1.5,0.5) {$\F^2$};
    \node at (2.5,0.5) {$\F$};
    \node at (0.5,1.5) {$\F$};
    \node at (2.5,1.5) {$0$};

    \draw[->, thick] (0,-0.5) -- (0,2.5) node[above] {$y$};
    \draw[->, thick] (-0.7,0) -- (3.5,0) node[right] {$x$};
    
    \draw[-, thin] (1,1) -- (2,1);
    \draw[-, thin] (2,0.3) -- (3.3,0.3);
    \draw[-, thin] (2,0.7) -- (3.3,0.7);
    \draw[-, thin] (1,0) -- (1,1);
    \draw[-, thin] (2,0.3) -- (2,2.3);

    \foreach \x in {0,1,2}
    \foreach \y in {0,0.3,0.7,1}
    {
        \pgfmathsetmacro\valueone{ (\x != 2.3) && ((\y == 0.3) || (\y == 1)) ? 1 : 0 }

        \pgfmathsetmacro\valuetwo{ (\x != 2.3)) && ((\y != 0.3)) ? 1 : 0 }

        \ifthenelse{\valueone=1 \AND\valuetwo =1}{
            \draw[color=black!80,fill=\pfivecolor!40, thick, xshift=0.3mm, yshift=0.3mm] (\x,\y) circle (\pointsize);
            \draw[color=black!80,fill=\pthreecolor!40, thick, xshift=-0.3mm, yshift=-0.3mm] (\x,\y) circle (\pointsize);
        }
    
        \ifthenelse{\valueone=1 \AND \valuetwo =0}{
            \draw[color=black!80,fill=\pfivecolor!40, thick] (\x,\y) circle (\pointsize);
        }{}

        \ifthenelse{\valuetwo=1 \AND \valueone = 0}{
            \draw[color=black!80,fill=\pthreecolor!40, thick] (\x,\y) circle (\pointsize);
        }{}

        \pgfmathparse{!\valueone && !\valuetwo}

        \ifthenelse{\pgfmathresult=1}{
            \draw[color=black!80,fill=\pointcolor!40, thick] (\x,\y) circle (\pointsize);
        }{}
    }

    \node at (-0.5,1) {\textcolor{blue!60}{$Q$}};
    
    \node at (1.5,-0.3) {$M_3$};
  \end{scope}
  \draw[->, thick] (2.5,-12) --node[anchor=north west] {$f_3$} (0.5,-14);

  \begin{scope}[scale= 2, shift={(0,0)}]
    
    \fill[\persistcolor!40] (0,0) rectangle (1,1);
    \fill[\persistcolor!40] (0,1) rectangle (1,2.3);
    \fill[\persistcolor!60] (1,0) rectangle (2,1);
    \fill[\persistcolor!40] (2,0) rectangle (3.3,1);
    \fill[\persistcolor!40] (0,1) rectangle (1,2.3); 
    \fill[\persistcolor!40] (1,1) rectangle (2,2.3);
    \fill[\persistcolor!20] (2,1) rectangle (3.3,2.3);

    \node (b) at (1.5,0.5) {$\F^2$};
    \node (c) at (2.5,0.5) {$\F$};
    \node (d) at (0.5,1.5) {$\F$};
    \node (f) at (2.5,1.5) {$0$};

    \draw[->, thick] (0,-0.5) -- (0,2.5) node[above] {$y$};
    \draw[->, thick] (-0.7,0) -- (3.5,0) node[right] {$x$};

    \draw[-, thin, color=\persistcolor!40] (0,1) -- (1,1);
    \draw[-, thin, color=\persistcolor!40] (1,1) -- (1,2.3);
    \draw[-, thin] (1,1) -- (3.3,1);
    \draw[-, thin] (1,0) -- (1,1);
    \draw[-, thin] (2,0) -- (2,2.3);

    \foreach \x in {0,1,2} 
    \foreach \y in {0,1}
    \draw[color=black!80,fill=\pfivecolor!40, thick] (\x,\y) circle (\pointsize) ;

    \node at (-0.5,1) {\textcolor{\pfivecolor!60}{$P$}};
    
    \node at (1.5,-0.3) {$M_4$};
  \end{scope}
\end{scope}
\end{tikzpicture}

\caption{A $\carte$-edit sequence between $\R^2$-persistence modules $M$ and $M_4$ from \Cref{ex:example_of_an_edit_sequence}. The right adjoints $g_i$, $i=1,2,3,4$ are order embeddings, and their images are indicated using colors.
}\label{fig:example_of_an_edit_sequence}
\end{figure}

\begin{example}\label{ex:example_of_an_edit_sequence}
   For $M$ given in \Cref{ex:edit_makes_simpler}, %
    we will construct a $\mathsf{Grid}_2$-edit sequence  \begin{equation}\label{eq:4-step edit}
        M\xleftarrow{f_1} M_1 \xrightarrow{f_2} M_2 \xleftarrow{f_3} M_3 \xrightarrow{f_4} M_4 \ \ \mbox{(see Figure~\ref{fig:example_of_an_edit_sequence})},
    \end{equation}
    whose cost attains the value of $\dedi^{\mathsf{Grid}_2}(M,M_4)$. Consider the $2$-grids
    \[P=\{0,1,2\}\times \{0,1\}, \ \ \ \ Q=\{0,1,2\}\times \{0,0.3,0.7,1\}, \ \ \ \  R=\{0,1,2\}\times\{0,0.5,1\}.\]
    $M$ and $M_4$ are $P$-presented, $M_1$ and $M_3$ are $Q$-presented, and $M_2$ is $R$-presented, and their restrictions to $P$, $Q$, and $R$, respectively, are given below.

{\[\begin{array}{cc}
\begin{tikzcd}[scale cd =\sizeone,ampersand replacement=\&,
               row sep=\sizetwo em, column sep=\sizethree em]
  {} \arrow[ddd, phantom, "M_1|_{Q}:" description] \&  \F \arrow[r,"1"] \&
  \F \arrow[r] \&  0 \\
  {} \&
  \F \arrow[u,"1"]  \arrow[r, "{\scalebox{0.70}{$\begin{pmatrix}
      1\\0
  \end{pmatrix}$}}"'] \&  \F^2 \arrow[r] \arrow[u,"(1\ 0)"'] \&
  0 \arrow[u] \\
  {} \&
  \F \arrow[u,"1"]  \arrow[r, "{\scalebox{0.70}{$\begin{pmatrix}
      1\\0
  \end{pmatrix}$}}"'] \&  \F^2 \arrow[r,"(1\ 1)"'] \arrow[u,"{\scalebox{0.95}{$\id_{\F^2}$}}"'] \&
  \F \arrow[u] \\
  {} \&
  \F \arrow[u,"1"]  \arrow[r, "{\scalebox{0.70}{$\begin{pmatrix}
      1\\0
  \end{pmatrix}$}}"'] \&  \F^2 \arrow[r,"{\scalebox{0.95}{$\id_{\F^2}$}}"'] \arrow[u,"{\scalebox{0.95}{$\id_{\F^2}$}}"'] \&
  \F^2 \arrow[u,"(1\ 1)"']
\end{tikzcd}&

\begin{tikzcd}[scale cd =\sizeone,ampersand replacement=\&,
               row sep=\sizetwo em, column sep=\sizethree em]
  {} \arrow[dd, phantom, "M_2|_{R}:" description] \&
  \F \arrow[r,"1"] \&
  \F \arrow[r] \&
  0 \\
  {} \&
  \F \arrow[u,"1"]
    \arrow[r, "{\scalebox{0.70}{$\begin{pmatrix} 1\\0 \end{pmatrix}$}}"'] \&
  \F^2 \arrow[r] \arrow[u,"(1\ 0)"'] \&
  0 \arrow[u]\\
  {} \&
  \F \arrow[u,"1"]
    \arrow[r, "{\scalebox{0.70}{$\begin{pmatrix} 1\\0 \end{pmatrix}$}}"'] \&
  \F^2 \arrow[r,"{\scalebox{0.95}{$\id_{\F^2}$}}"'] \arrow[u,"{\scalebox{0.95}{$\id_{\F^2}$}}"'] \&
  \F^2 \arrow[u]
\end{tikzcd}
\\[7em]
\begin{tikzcd}[scale cd =\sizeone,ampersand replacement=\&,
               row sep=\sizetwo em, column sep=\sizethree em]
  {} \arrow[ddd, phantom, "M_3|_{Q}:" description] \&  \F \arrow[r,"1"] \&
  \F \arrow[r] \&  0 \\
  {} \&
  \F \arrow[u,"1"]  \arrow[r, "{\scalebox{0.7}{$\begin{pmatrix}
      1\\0
  \end{pmatrix}$}}"'] \&  \F^2 \arrow[r] \arrow[u,"(1\ 0)"'] \&
  0 \arrow[u] \\
  {} \&
  \F \arrow[u,"1"]  \arrow[r, "{\scalebox{0.7}{$\begin{pmatrix}
      1\\0
  \end{pmatrix}$}}"'] \&  \F^2 \arrow[r,"(0\ 1)"'] \arrow[u,"{\scalebox{0.95}{$\id_{\F^2}$}}"'] \&
  \F \arrow[u] \\
  {} \&
  \F \arrow[u,"1"]  \arrow[r, "{\scalebox{0.7}{$\begin{pmatrix}
      1\\0
  \end{pmatrix}$}}"'] \&  \F^2 \arrow[r,"{\scalebox{0.95}{$\id_{\F^2}$}}"'] \arrow[u,"{\scalebox{0.95}{$\id_{\F^2}$}}"'] \&
  \F^2 \arrow[u,"(0\ 1)"']
\end{tikzcd}
&
\begin{tikzcd}[scale cd =\sizeone,ampersand replacement=\&,
               row sep=\sizetwo em, column sep=\sizethree em]
  {} \arrow[d, phantom, "M_4|_{P}:" description] \&
  \F \arrow[r,"1"] \&
  \F \arrow[r] \&
  0 \\
  {} \&
  \F \arrow[u,"1"]
    \arrow[r, "{\scalebox{0.7}{$\begin{pmatrix} 1\\ 0 \end{pmatrix}$}}"'] \&
  \F^2 \arrow[r,"(0\ 1)"'] \arrow[u,"(1\ 0)"'] \&
  \F \arrow[u]
\end{tikzcd}
\end{array}\]}
  Let $g_1,g_4:P\hookrightarrow Q$ and $g_2,g_3:R\hookrightarrow Q$ be the grid morphisms given by
        \begin{align*}
           g_1=g_4= \id_{\{0,1,2\}}\times \begin{pmatrix}
                0\mapsto 0.3\\
                1\mapsto \phantom{.}1\phantom{3}
            \end{pmatrix} \  \ \ \  \mbox{and}\ \ \ \ \ 
           g_2=g_3= \id_{\{0,1,2\}}\times \begin{pmatrix}
                \phantom{5}0\phantom{.}\mapsto 0\\
                0.5\mapsto0.7 \\
                \phantom{5}1\phantom{.}\mapsto \phantom{.}1\phantom{3}
            \end{pmatrix}.
        \end{align*}
   For each $i=1,2,3,4$, let $f_i$ be the left adjoint of $g_i$, which appears in the sequence given in \eqref{eq:4-step edit}. %
   Then the left adjoints $f_i$ are those described by \Cref{prop:existence of adjoint}~\ref{item:existence of adjoint2}.
    The cost of this edit sequence is $\dist{f_1}+\dist{f_2}+\dist{f_3}+\dist{f_4}=0.3+0.2+0.3+0.2=1$. In fact, the constructed $\mathsf{Grid}_2$-sequence attains the value of $\dedi^{\mathsf{Grid}_2}(M,M_4)$: 
This follows from the fact that $\dint(M,M_4)=1$ (which is not difficult to show) and from \Cref{thm:dedi=dint}.
\end{example}

\section{Interleaving distance as a Galois-edit distance}\label{sec:interleaving as edit}
 The main goal of this section is to establish: 

\begin{theorem}[Interleaving distance as a Galois-edit distance]\label{thm:dedi=dint} 
For any finitely presented $\R^d$-persistence modules $M$ and $N$, we have   \(\dint(M,N)=\dedi^\calE(M,N).\) %
\end{theorem}

In \Cref{appendix:comparison with gulen's work}, we compare \Cref{thm:dedi=dint} with related ideas
that have appeared in the literature. Now, we discuss some implications of \Cref{thm:dedi=dint}.
Since $\dint$ is an extended metric (not merely a pseudometric) on finitely presented $\R^d$-persistence modules \cite[Corollary 6.2]{lesnick2015theory}, 
by \Cref{thm:dedi=dint}, so is $\dedi^\calE$. We also remark that \Cref{thm:dedi=dint} and the main theorem of \cite{bjerkevik2020computing} imply that $\dedi^\calE$ is NP-hard to compute.  

Given any categories $\calC$ and $\calD$, we call a map $\mathfrak{f}:\ob(\calC)\to\ob(\calD)$ an \textbf{invariant} if,  for any $M,N\in \ob(\calC)$, an isomorphism $M\cong N$ implies an isomorphism $\mathfrak{f}(M)\cong \mathfrak{f}(N)$. 

\begin{corollary}\label{cor:one stability implies another} %
Let $\mathfrak{f}$ be any invariant of %
finitely presented $\R^d$-persistence modules, and let $\sfd$ be a metric on the image of $\mathfrak{f}$.  The invariant $\mathfrak{f}$ is stable (i.e., continuous) for the topologies induced by $\sfd$ and $\dint$ if and only if  $\mathfrak{f}$ is stable for the topologies induced by $\sfd$ and $\dedi^\calE$. 
\end{corollary}

The following proposition is useful  for establishing \Cref{thm:dedi=dint},  for strengthening \Cref{cor:one stability implies another}, and for  an alternative proof of the well-known interleaving-bottleneck stability theorem (cf.~\Cref{thm:bottleneck stability}). 

It says that any distance $\sfd$ that is \textbf{1-Lipschitz under a single $\calE$-edit} (i.e., if there is an $\calE$-\edit{} $f$ from $M$ to $N$, then $\sfd(M,N)\le \dist{f}$) does not exceed the edit distance. 

\begin{proposition}[Maximality of the edit distance]\label{prop:edit_is_universal}
    Let $\sfd$ be any pseudometric on finitely presented $\R^d$-persistence modules. Suppose $\sfd$ is 1-Lipschitz under every single $\calE$-edit. Then, for any finitely presented $\R^d$-persistence modules $M$ and $N$, we have that $\sfd(M,N)\le \dedi^\calE(M,N)$. 
\end{proposition}

We remark that the preceding proposition is an instance of the maximality of edit distances in general. Namely, let $\mathcal{O}$ be a collection of objects and consider a set of allowed edits between them. If $d$ is any metric on $\mathcal{O}$ such that $d(o_1, o_2)$ is at most the cost of any single edit $o_1 \to o_2$, then $d$ is bounded above by the induced edit distance. The proof presented below is straightforwardly extendable to this general setting, as can be found, for example, in \cite[p.~10]{carlsson2010characterization}.

\begin{proof}[Proof of \Cref{prop:edit_is_universal}]
If $\dedi^\calE(M,N)=\infty$, then there is nothing to prove. Suppose that $\dedi^\calE(M,N)<\infty$. Then, for any $\ep>0$, there exists a sequence of $\R^d$-persistence modules $M=M_0,\ldots, M_n=N$ for some $n$ and $\calE$-\edit{}s $f_i:M_i\leftrightarrow M_{i+1}$ such that $\sum_{i=0}^{n-1}\dist{f_i}<\dedi^\calE(M,N)+\ep$. Then, we have:
    \begin{align*}
        d(M,N)&\le \sum_{i=0}^{n-1} d(M_i,M_{i+1})&&\mbox{by the triangle inequality}\\
        &\le \sum_{i=0}^{n-1} \dist{f_i}&&\mbox{by assumption} 
        \\
        &<\dedi^\calE(M,N)+\ep.&&
    \end{align*}
    Since $\ep>0$ was arbitrary, we obtain the desired inequality.
\end{proof}

A consequence of the preceding proposition is that, for any metric $\sfd$ on the image of an invariant $\mathfrak{f}$ of finitely presented $\R^d$-persistence modules, the 1-Lipschitzness of $\mathfrak{f}$ with respect to $\sfd$ and $\dint$ is equivalent to its 1-Lipschitzness with respect to $\sfd$ and a single $\calE$-edit:

\begin{corollary}
\label{cor:single-edit is enough}Let $\mathfrak{f}$ be any invariant of finitely presented $\R^d$-persistence modules, and let $\sfd$ be a metric on the image of $\mathfrak{f}$. 
Then, $d\circ(\mathfrak{f}\times\mathfrak{f})\le \dint$ if and only if  $d\circ(\mathfrak{f}\times\mathfrak{f})$ is 1-Lipschitz under every single $\calE$-edit.
\end{corollary}

Now, we outline the proof of \Cref{thm:dedi=dint}. The key ingredient is \Cref{thm:Galois_edit_is_presentation_edit}, 
 which states that a Galois-edit between persistence modules corresponds, in a suitable sense, to a notion of Galois-edit between their presentations.
This theorem reduces the problem of proving the inequality $\dint \le \dedi^{\Poset}$ to that of showing that, if there is a Galois-edit of cost $\ep\geq 0$ between two presentations, then the corresponding $\R^d$-persistence modules are $\ep$-interleaved; this is what we establish.
Next, we prove the inequality~$\dedi^\carte \le \dint$. For this, we invoke \Cref{lesnickpresentation}, which ensures that two $\ep$-interleaved persistence modules admit presentations on common symbol sets and consequently let us construct a Galois-edit sequence between their presentations. We then again employ the correspondence from \Cref{thm:Galois_edit_is_presentation_edit} to obtain a $\grid$-edit sequence of cost at most $\ep$.  %
In sum, since $\dedi^\Poset\leq \dedi^\carte$ (\Cref{relationbetweenedits}~\ref{item:relationbetweenedits1}), we have that 
 $\dint\le \dedi^\Poset \le \dedi^{\carte}\le \dint,$
 completing the proof of \Cref{thm:dedi=dint}.

Before \Cref{thm:Galois_edit_is_presentation_edit}, we remark that, given a presentation of an $\R^d$-persistence module $M$,  its homogeneous parts $M_p$ and structure maps $M_{p\leq q}$ can be described in terms of the presentation: For a graded set $W$ and a set $A\subset \R^d$, we write $W^A$ to denote $\{w\in W: \gr_W(w)\in A\}$. %
\begin{remark} \label{rem:vector space of modules} %
   Let $\left<W|Y\right>$ be a presentation of an $\R^d$-persistence module $M$, and $p\in \R^d$.  %
   Then,   $Y^{p\bindownarrow}$ is a homogeneous subset of $\fr[W^{p\bindownarrow}]$ and
   \[M_p\cong \left<W|Y\right>_p \cong\langle W^{p\bindownarrow}|Y^{p\bindownarrow}\rangle_p \cong \left<\sym(W^{p\bindownarrow})|\sym(Y^{p\bindownarrow})\right>_\F.\]   
    Now, let $p\le q$ in $\R^d$. Since $\sym(W^{p\bindownarrow})\subseteq \sym(W^{q\bindownarrow})$ and $\sym(Y^{p\bindownarrow})\subseteq \sym(Y^{q\bindownarrow}) \subseteq \fr_\F[W^{q\bindownarrow}]$, we obtain the canonical maps %
    \[M_p\cong \left<\sym(W^{p\bindownarrow}) |\sym(Y^{p\bindownarrow})\right>_\F{\hookrightarrow} \left<\sym(W^{q\bindownarrow})|\sym(Y^{p\bindownarrow})\right>_\F {\twoheadrightarrow} \left<\sym(W^{q\bindownarrow})|\sym(Y^{q\bindownarrow})\right>_\F\cong M_q.\] 
    The composition of these maps is $M_{p\le q}$ (equivalently, the left multiplication by $x^{q-p}$).
\end{remark}

\begin{theorem}[Galois-edits between persistence modules correspond to Galois-edits between their presentations]\label{thm:Galois_edit_is_presentation_edit}
    Let 
    $f:P\to Q$ be a morphism in $\calE$, 
    $M$ and $N$ be $P$-,$Q$-presented $\R^d$-persistence modules. 
    Let $\langle W|Y\rangle$ be a presentation of $M$ over $P$.
    Then, the following are equivalent:
    \begin{itemize}
        \item $f$ is an $\calE$-edit from $M$ to $N$,
        \item $N\cong  \langle f(W)|f(Y)\rangle$. %
    \end{itemize}
\end{theorem}

For a poset $Q\subset \R^d$ and $a\in \R^d$, let $Q_{\le a}:= Q \cap a\bindownarrow \;=\; \{\, q \in Q : q \le a \,\}$.

\begin{proof}

    We fix $a\in \R^d$ throughout the proof. Let $f$ be an $\calE$-edit from $M$ to $N$ with a right adjoint  $g:Q\to P$, and let $g(Q_{\leq a})$ be the image of $Q_{\leq a}$ under $g$.
    Consider \[g(Q_{\le a})\bindownarrow:=\{p\in P: \mbox{there exists b $\in g(Q_{\le a})$ such that $p\leq b$}\}.\] Let $\bigvee  Q_{\le a}$ be the join of $Q_{\le a}$ in $\R^d$. Then, we have:
\begin{equation}\label{eq:sequence of iso}
    \hspace{-2cm}\begin{aligned}
    N_a&\cong \colim\, N|_{Q_{\le a}} &&\mbox{since $N$ is $Q$-presented (cf. the paragraph before \Cref{KanExtensionOfPersistenceModules})}
    \\& \stackrel{\mathclap{\normalfont\mbox{$(\ast)$}}}{\cong} \colim \left((M|_P)\circ (g|_{Q_{\le a}})\right) &&\mbox{since $M|_P\circ g\cong N|_Q$}
    \\&
    \cong\colim\, M|_{g(Q_{\le a})\bindownarrow} && \mbox{%
    \textcolor{darkgreen}{(i)}}
    \\&\cong \langle \sym(W^{g(Q_{\le a})\bindownarrow})|\sym(Y^{g(Q_{\le a})\bindownarrow})\rangle_\F && \mbox{%
    \textcolor{darkgreen}{(ii)}}
    \\&\cong\langle \sym(f(W)^{ Q_{\le a}})|\sym(f(Y)^{Q_{\le a}})\rangle_\F && \mbox{\textcolor{darkgreen}{(iii)}}
    \\&\cong\langle f(W)|f(Y)\rangle_{\bigvee Q_{\le a}}  && \mbox{\textcolor{darkgreen}{(iv)}} %
    \\&=\langle f(W)|f(Y)\rangle _a.
    \end{aligned}\hspace{-2cm}
    \end{equation}
\begin{enumerate}[leftmargin=*,label=\textcolor{darkgreen}{(\roman*)}]
    \item 
       By \Cref{lem:final_functor_preserves_colimit}, it suffices to prove that the monotone function $g|_{Q\le a}:Q_{\le a}\to g(Q_{\le a})\bindownarrow$  
    is final. Let $p\in g\left(Q_{\le a}\right)\bindownarrow$. We wish to show that the poset $(Q_{\le a})_{g\geq p}=\{q\in Q_{\le a}: g(q) \ge p\}$ %
    is connected. Suppose $x\in (Q_{\le a})_{g\geq p}$, i.e., $p\le g(x)$ and $x\le a$.  By \Cref{productofadjoints}~\ref{item:fg is deflating and gf is inflating}, we have $f\circ g \le \id_Q$ and thus
    \[f(p) \le f(g(x))\le x \le a.\] By the same remark, we have $g(f(p))\ge p$, and thus $f(p)\in (Q_{\le a})_{g\geq p}$.
    This shows that every element $x$ of the poset $(Q_{\le a})_{g\geq p}$ is comparable with $f(p)\in (Q_{\le a})_{g\geq p}$, and thus  $(Q_{\le a})_{g\geq p}$ is connected.
    \item 
   
     Consider the poset category $\calC=g(Q_{\le a})\bindownarrow$ and the category $\calD=\left\{0\displaystyle\mathrel{\mathop{\rightrightarrows}^{f}_{g}}\textstyle 1\right\}$ of two objects and two non-identity arrows. We will define a functor $F:\calC\times \calD \to \vect$. %
     For any $p\le p'\in \ob(\calC)$, let $F$ be defined by
     \begin{align*}
         (p,0) &\mapsto \fr_\F[\sym(Y^{p\bindownarrow})] \\
         (p,1) & \mapsto \fr_\F[\sym(W^{p\bindownarrow})]\\
         (p\le p', 0) & \mapsto\mbox{the inclusion } \fr_\F[\sym(Y^{p\bindownarrow})]\hookrightarrow \fr_\F[\sym(Y^{p'\bindownarrow})] \\
         (p\le p', 1) & \mapsto\mbox{the inclusion } \fr_\F[\sym(W^{p\bindownarrow})]\hookrightarrow \fr_\F[\sym(W^{p'\bindownarrow})] \\
         (p,f) & \mapsto  \mbox{the linear map induced by the map } Y^{p\bindownarrow} \to \fr_\F[\sym(W^{p\bindownarrow})], \\& \hspace{5cm} \left(\sum_i c_iw_i,b\right)\mapsto \sum_i c_iw_i, \\
         & \hspace{4cm}\text{ where $c_i\in \F\setminus \{0\}$, $w_i\in W^{p\bindownarrow}$, and $b\in p\bindownarrow$}\\
         (p,g) & \mapsto 0 \mbox{ (the zero map)}.
     \end{align*}
   Indeed, this specifies the value of $F$ for each object and non-identity morphism, and the values of $F$ for identity morphisms are trivially defined.
   Then, for any pair $(r:p\le p', s:d\to d')$ of morphisms in $\calC$ and $\calD,$ the equation $F(p',s)\circ F(r,d)=F(r,d')\circ F(p,s)$ holds. Therefore, we have defined a functor $F:\calC\times \calD\to \vect$ (cf. \cite[II.3 Proposition 1]{mac2013categories}).

    For each $p\in \ob(\calC)$, the colimit of the diagram $F|_{\{p\}\times \calD}:\{p\}\times \calD \to \vect$ is the presentation $\langle \sym(W^{p\bindownarrow}) |\sym( Y^{p\bindownarrow})\rangle_\F\cong M_p$, and %
    for any $p\le p' \in \ob(\calC)$, the map \[\colim\, F|_{\{p\}\times \calD}=\langle \sym(W^{p\bindownarrow}) | \sym(Y^{p\bindownarrow})\rangle_\F\to \langle \sym(W^{p'\bindownarrow}) | \sym(Y^{p'\bindownarrow})\rangle_\F=\colim\, F|_{\{p'\}\times \calD}\] is induced by the inclusions $\fr_\F[\sym(W^{p\bindownarrow})]\hookrightarrow \fr_\F[\sym(W^{p'\bindownarrow})]$ and $\fr_\F[\sym(Y^{p\bindownarrow})]\hookrightarrow \fr_\F[\sym(Y^{p'\bindownarrow})]$%
    . By \Cref{rem:vector space of modules}, $\colim_d F(-,d):\calC\to \vect$ is naturally isomorphic to $M|_{g(Q_{\leq a})\bindownarrow}$.
    
    For each $d\in \ob(\calD)$, the colimit of the diagram $F|_{\calC\times \{d\}}:\calC\times \{d\}\to \vect$ is
    \[
\begin{cases}
\fr_\F\!\left[\sym\!\left(Y^{g(Q_{\le a})\bindownarrow}\right)\right], & \text{if } d = 0, \\[4pt]
\fr_\F\!\left[\sym\!\left(W^{g(Q_{\le a})\bindownarrow}\right)\right], & \text{if } d = 1.
\end{cases}
\]
    The functoriality of the correspondence $d\mapsto \colim\, F|_{\calC\times\{d\}}$ gives two morphisms from $\fr_\F[\sym(Y^{g(Q_{\le a})\bindownarrow})]$ to $\fr_\F[\sym(W^{g(Q_{\le a})\bindownarrow})]$, which are $\colim_c F(c,f)$ and $\colim_c F(c,g)$. Namely, $\colim_c F(c,f)$ is the natural map induced by the map \begin{align*}
        Y^{g(Q_{\le a})\bindownarrow}\to \fr_\F[\sym(W^{g(Q_{\le a})\bindownarrow})], \quad \quad & \text{given by }\left(\sum_i c_iw_i,b\right)\mapsto \sum_i c_iw_i,\\
        & \text{where }c_i\in \F, w_i\in W^{g(Q_{\le a})\bindownarrow},\text{ and }b\in g(Q_{\le a}).
    \end{align*} and $\colim_c F(c,g)$ is the zero map. Therefore, the colimit of the diagram $\colim_c F(c,-):\calD \to \vect$ is the cokernel of $\colim_c F(c,f)$, %
    which is the presentation $\langle \sym(W^{g(Q_{\le a})\bindownarrow} )| \sym(Y^{g(Q_{\le a})\bindownarrow})\rangle_\F$.\\
    By \Cref{lem:colimt_commutes_with_colimit}, the colimit of the diagram $\colim_d F(-,d) \cong M|_{g(Q_{\leq a})\bindownarrow}:\calC \to \vect$ 
     is \[\langle \sym(W^{g(Q_{\le a})\bindownarrow} )| \sym(Y^{g(Q_{\le a})\bindownarrow})\rangle_\F.\]
    
    \item 
    It suffices to prove that for any $p\in \gr(W\sqcup Y) \subset  P$, $p\in g(Q_{\le a})\bindownarrow$ if and only if $f(p)\in Q_{\le a}$. If $p\in g(Q_{\le a})\bindownarrow$, then $p\le g(q)$ for some $q\in Q_{\le a}$. Then $f(p)\le f(g(q))\le q \le a$ by \Cref{productofadjoints}~\ref{item:fg is deflating and gf is inflating}, thus $f(p)\in Q_{\le a}$. Conversely, if $f(p)\in Q_{\le a}$, then we have $p\le g(f(p))\in g(Q_{\le a})$ by the same remark, and thus we have $p\in g(Q_{\le a})\bindownarrow$.
    \item 
    It suffices to show that for any $b\in \gr(f(W)\sqcup f(Y))\subset Q$, $b\le \bigvee Q_{\le a}$ if and only if $b\in Q_{\le a}$. If  $b\le \bigvee Q_{\le a}$. Then, $b \le \bigvee Q_{\le a}\le a$, thus $b\in Q_{\le a}$. The converse follows by the definition of the join $\bigvee Q_{\le a}$. %
\end{enumerate}

\vspace{2mm}

    Since every isomorphism in \eqref{eq:sequence of iso} is canonical, the isomorphisms $N_{a}\cong \langle f(W)|f(Y)\rangle _{a}$ commutes with the structure maps. Therefore,
    we obtain the isomorphism $N\cong \langle f(W)|f(Y)\rangle$.

    Conversely, assume that $N\cong \langle f(
    W)|f(Y)\rangle$. Then we have an isomorphism between the first and the last terms \eqref{eq:sequence of iso}. In addition, all isomorphisms in \eqref{eq:sequence of iso} hold except for $(\ast)$.  Therefore, the isomorphism $(\ast)$ also holds. 
    Now, we show that $M|_{P}\circ g \cong N|_Q$.
    Let $a\in Q$. As $Q_{\le a}=a\bindownarrow$ in $Q$, we have $g(Q_{\le a})\bindownarrow = g(a\bindownarrow)\bindownarrow=g(a)\bindownarrow$ in $P$. Therefore, $N_a\cong \colim\, M|_{g(a)\bindownarrow} \cong M_{g(a)}$. Since every isomorphism of \eqref{eq:sequence of iso} is canonical, the isomorphisms $N_{(a)}\cong M_{g(a)}$ commute with the structure maps. Hence, we have $M|_{P}\circ g \cong N|_Q$ and $f$ is an $\calE$-edit from $M$ and $N$.
\end{proof}

\begin{example}\label{ex:presentation_example}%
    From \Cref{ex:edit_makes_simpler}, recall the $\mathsf{Grid}_2$-edit $f:M\to N$. We have $M\cong \langle W_1|Y_1\rangle$ where
    \begin{align*}
        W_1&:=\left\{(w_1,(0,0)),(w_2,(1,0))\right\}\\
        Y_1&:=\left\{(w_2,(1,1)),(w_1-w_2,(2,0))\right\}.
    \end{align*}
    By \Cref{thm:Galois_edit_is_presentation_edit}, with the \edit{} $f$, we obtain a presentation $N \cong \langle f(W_1) | f(Y_1)\rangle$ where
    \begin{align*}
        f(W_1)&:=\left\{(w_1,(0,0)),(w_2,(2,0))\right\}\\
        f(Y_1)&:=\left\{(w_2,(2,1)),(w_1-w_2,(2,0))\right\}.
    \end{align*}
    The presentations of $M$ and $N$ described above are illustrated in \Cref{fig:example_of_a_presentation_edit}.

\begin{figure}
\centering
\begin{tikzpicture}[scale=1.2, every node/.style={scale=0.9}]
\tikzset{
  axis/.style={->, very thin},
  box/.style={draw=black!40, very thin},
  genpoint/.style={circle, thin,inner sep=2pt,fill=blue!100,draw=black},
  relpoint/.style={circle, thin ,inner sep=2pt,fill=red!100,draw=black}
}

\begin{scope}[scale=1, shift={(0,0)}]

  \fill[\persistcolor!40] (0,0) rectangle (1,1);
  \fill[\persistcolor!40] (0,1) rectangle (1,2.3);
  \fill[\persistcolor!60] (1,0) rectangle (2,1);
  \fill[\persistcolor!40] (2,0) rectangle (3.3,1);
  \fill[\persistcolor!40] (1,1) rectangle (2,2.3);
  \fill[\persistcolor!20] (2,1) rectangle (3.3,2.3);

  \node at (1.5,0.5) {$\F^2$};
  \node at (2.5,0.5) {$\F$};
  \node at (0.5,1.5) {$\F$};
  \node at (2.5,1.5) {$0$};

  \draw[->, thick] (0,-0.5) -- (0,2.5) node[above] {$y$};
  \draw[->, thick] (-0.7,0) -- (3.5,0) node[right] {$x$};

  \draw[-, thin, color=\persistcolor!40] (0,1) -- (1,1);
  \draw[-, thin, color=\persistcolor!40] (1,1) -- (1,2.3);
  \draw[-, thin] (1,1) -- (3.3,1);
  \draw[-, thin] (1,0) -- (1,1);
  \draw[-, thin] (2,0) -- (2,2.3);

  \node[genpoint,label=below left:$w_1$] at (0,0) {};
  \node[genpoint,label=below:$w_2$] at (1,0) {};

  \node[relpoint,label=left:$w_2$] at (1,1) {};
  \node[relpoint,label=below:$w_1-w_2$] at (2,0) {};

  \node at (-0.5,1) {$P$};

\end{scope}

\begin{scope}[shift={(5,0)}]

  \fill[red!40] (0,0) rectangle (2,2.3);
  \fill[red!40] (0,0) rectangle (3.3,1);
  \fill[red!20] (2,1) rectangle (3.3,2.3);

  \node at (1,0.8) {$\F$};
  \node at (2.5,1.5) {$0$};

  \draw[->, thick] (0,-0.5) -- (0,2.5) node[above] {$y$};
  \draw[->, thick] (-0.7,0) -- (3.5,0) node[right] {$x$};

  \draw[-, thin] (2,1) -- (3.3,1);
  \draw[-, thin] (2,1) -- (2,2.3);

  \node[genpoint,label=below left:$w_1$] at (0,0) {};
  \node[genpoint,label=above right:$w_2$] at (2.03,0.03) {};

  \node[relpoint,label=left:$w_2$] at (2,1) {};
  \node[relpoint,label=below:$w_1-w_2$] at (1.97,-0.03) {};
  \node at (-0.5,1) {$Q$};

\end{scope}

\node at (1.5,-0.7) {$M$};
\node at (6.5,-0.7) {$N$};

\end{tikzpicture}

\caption{Illustration of $M$ and $N$ with their presentations from Example~\ref{ex:presentation_example}. Blue dots represent generators, and red dots represent relations.}
\label{fig:example_of_a_presentation_edit}

\end{figure}
\end{example}

\begin{corollary}\label{cor:edit_by_monotone_function}
    Let $M\cong \langle W_1|Y_1\rangle$ and $N\cong\langle W_2|Y_2\rangle$ be finite presentations of $\R^d$-persistence modules on the common symbol sets, i.e., $W=\sym(W_1)=\sym(W_2)$ and $Y=\sym(Y_1)=\sym(Y_2)$.\footnote{We remark that $M$ and $N$ admit presentations on %
common symbol sets if and only if  $M$ and $N$ have \emph{presentation} \emph{matrices with the same underlying matrix} %
\cite[Definition~3.1]{bjerkevik2021ell}.} %
    Let $P$ and $Q$ be grids containing $\gr_M(X)$ and $\gr_N(X)$, respectively. Let $h:P\to Q$ be a grid morphism. If $h\circ \gr_M = \gr_N$ and $h$ admits a right adjoint, then $h$ is a $\grid$-edit from $M$ to $N$.
\end{corollary}
\begin{proof}
    Since $h\circ \gr_M = \gr_N$, we have that $N\cong \langle h(W_1)|h(Y_1)\rangle$. Now, by \Cref{thm:Galois_edit_is_presentation_edit}, $h$ is a $\grid$-edit from $M$ to $N$.
\end{proof}

\begin{lemma}\label{lem:edit_induces_lesnick_interleaving'}
    Let $f:P\to Q$ be a monotone function, and suppose the grade sets of $W_1$ and $Y_1$ are contained in $P$.
    Then $\langle W_1 | Y_1\rangle $ and $\langle f(W_1) | f(Y_1)\rangle$ are $\dist{f}$-interleaved.
\end{lemma}

This lemma appears in the literature; see \cite[Theorem 4.4]{lesnick2015theory} and the proof of \cite[Theorem 3.7]{bjerkevik2021ell}. However, we have not found any written proof of it, and thus we include a proof in \Cref{appendix:omitted proofs}.

\begin{proposition} 
    \label{prop:interleavingedit}
    For any finitely presented $\R^d$-persistence modules $M$ and $N$, we have that
    $ d_I(M,N)\le \dedi^{\Poset}(M,N).$ 
\end{proposition}

\begin{proof}
Assume that $P,Q\in \ob(\Poset)$, and $M$ and $N$ are $P$- and $Q$-presented (\Cref{def:alternative definition of finitely presented over P}), and $f:P\to Q$ is a $\Poset$-\edit{} $M\rightarrow N$.   By \Cref{prop:edit_is_universal}, it suffices to show that $\dint(M,N)\leq \dist{f}$.
    Since $M$ is $P$-presented, there exists a finite presentation $M\cong\langle W_1|Y_1\rangle$ such that the grade sets of $W_1$ and $Y_1$ are contained in $P$.
    By \Cref{thm:Galois_edit_is_presentation_edit}, 
    we have $N\cong \langle f(W_1)|f(Y_1)\rangle$. %
    Now the claim directly follows from \Cref{lem:edit_induces_lesnick_interleaving'}.
\end{proof}

Note that the inequality $d_I\le \dedi^{\Poset}$ implies $d_I\le \dedi^{\JsLat}$. Nevertheless, in \Cref{appendix:another proof}, we provide a separate proof of the latter inequality that does not rely on \Cref{thm:Galois_edit_is_presentation_edit} or \Cref{lem:edit_induces_lesnick_interleaving'}. Instead, that proof adapts ideas from \cite{gulen2022galois} and is order-theoretic.

Next, we outline the proof of the inequality $\dedi^{\grid} \le d_I$. 
 Note that, by \Cref{lesnickpresentation}, if $\R^d$-persistence modules $M$ and $N$ are $\ep$-interleaved for some $\ep>0$, they admit presentations on common symbol sets. %
Let $\gr_M$ and $\gr_N$ be their grading functions. It is not difficult to see that, by utilizing \Cref{ex:presentations_are_valid} and \Cref{lesnickpresentation}, the interpolation $\{(1-t)\gr_M+t\gr_N\}_{t\in[0,1]}$ provides a family of presentations over the common symbol sets, and that this family is a geodesic path between $M$ and $N$. Meanwhile, recall that \Cref{cor:edit_by_monotone_function} provides a sufficient condition for a grid morphism to be a $\grid$-edit between presentations on common symbol sets. \Cref{cor:edit_by_monotone_function} together with the two lemmas below enables us to construct a $\grid$-edit sequence of persistence modules 
that appear along the geodesic path, whose cost does not exceed $\ep$. This implies the inequality $\dedi^{\grid} \le d_I$. %

For the next lemma, we introduce the following definition. The \textbf{injectivity radius} of a function $f:W\to \R^d$ is defined to be
    \[\mathrm{inj}(f):=\min\left\{\frac{1}{2}\abs{f(w)_i-f(w')_i}:w,w'\in P,\ i\in [d] , f(w)_i\neq f(w')_i
 \right\}.\]
In words, the injectivity radius of $f$ is half the minimal the $\infty$-norm distance between two \emph{distinct} points of the smallest grid containing $f(W)$.
This definition extends in a straightforward manner to any finite \emph{graded} set $W$, by taking the injectivity radius of its grading function, i.e.,
$\inj(W):=\inj(\gr_W(W))$.

\begin{lemma}[Easy-edit lemma]\label{lem:localstability} %
    Let $M\cong \langle W_1|Y_1\rangle$ and $N\cong\langle W_2|Y_2\rangle$ be finite presentations of $\R^d$-persistence modules on the common symbol sets, i.e., $W=\sym(W_1)=\sym(W_2)$ and $Y=\sym(Y_1)=\sym(Y_2)$. For the disjoint union $X:=W\sqcup Y$, we have the grading functions $\gr_M, \gr_N:X\to \R^d$ for $M$ and $N$, respectively. 
    If $\norm{\gr_M-\gr_N}_\infty< \inj (\gr_N)$, then there exists a $\grid$-edit $f:M\to N$ such that $\dist{f}\le \norm{\gr_M-\gr_N}_\infty$. %
\end{lemma}

\begin{proof}[Proof]
Let $\ep:=\norm{\gr_M-\gr_N}_\infty$.
    Let $P$ be the smallest grid containing the set $\left(\gr_N (W\sqcup Y)+\ep\ones\right)\cup \gr_M (W\sqcup Y)$, and let $Q$ be the smallest grid containing the set $%
    \gr_N(W\sqcup Y)$ (cf. \Cref{rem:smallest grid}). 
    By %
    \Cref{rem:constructible_over_grid_closure},
    $M$ is $P$-presented, and $N$ is $Q$-presented. By \Cref{cor:edit_by_monotone_function}, it suffices to construct a monotone function $f:P\to Q$ such that 
$f\circ \gr_M=\gr_N$.

We define $f$ by sending each $p\in P$ to the \emph{unique} point $q\in Q$ such that $\norm{p-q}_{\infty}\le \ep$.

\noindent\emph{Claim 1.} $f$ is well-defined: such a point $q\in Q$ exists and is unique. 

   \begin{proof}
   Fix $p\in P$, and 
   we first prove the existence of $q\in Q$. By the definition of $Q$ (cf. \Cref{rem:smallest grid}), it suffices to show that for each $i\in [d]$, there exists unique $q_i\in \pi_i\left(\gr_N(W\sqcup Y)\right)$ such that $\abs{p_i-q_i}\le \ep$.\\
   Let $i\in[d]$. By the definition of $P$, $p_i$ is either in $\pi_i\left(\gr_N (W\sqcup Y)+\ep\ones\right)$ or in $\pi_i \left(\gr_M(W\sqcup Y)\right)$. 
   \begin{enumerate}[label=(\roman*)]
       \item If $p_i\in \pi_i\left(\gr_N (W\sqcup Y)+\ep\ones\right)=\pi_i\left(\gr_N(W\sqcup Y)\right) + \ep$, then we take $q_i:=p_i-\ep \in \pi_i\left(\gr_N(W\sqcup Y)\right)$.
       \item If $p_i \in \pi_i \left(\gr_M(W\sqcup Y)\right)$, then there exists $w\in W\sqcup Y$ such that $\gr_M(w)_i=p_i$. As $\ep = \norm{\gr_M-\gr_N}_\infty$, $\abs{p_i-\gr_N(w)_i}\le \ep$. Thus we take $q_i:=\gr_N(w)_i \in \pi_i \left(\gr_N(W\sqcup Y)\right).$
   \end{enumerate}
    By construction given above, $f$ is a grid morphism. And the uniqueness part of \emph{Claim 1} %
    guarantees that $f$ does not violate the monotonicity. 
    
   For the uniqueness, suppose that $q,q'\in Q$ are within distance $\ep$ from $p$. By the triangle inequality, $\lVert q-q'\rVert_\infty \le 2\ep$. Hence, for each $i\in [d]$, we have that
   \[\frac{1}{2}\abs{q_i-q_i'}\le\ep=\norm{\gr_M-\gr_N}_\infty<\inj(\gr_N).\]
   Then, by the definition of $\inj(\gr_N)$, we have that $q_i=q_i'$, and thus $q=q'$, as desired.   
   \end{proof}

   By the uniqueness property and that $\norm{p-f(p)}_\infty\le \ep$ for all $p\in P$, for any $x\in X$, we have $f(\gr_M(x))=\gr_N(x)$. Thus, $f\circ \gr_M = \gr_N$.

    To apply \Cref{cor:edit_by_monotone_function}, we lastly show that $f$ has a right adjoint, which can be shown by proving that $f$ preserves the bottom (see the discussion after \Cref{def:differentclassesforpersistencemoduledomain}). Let $\bot_Q$ be the bottom of $Q$. As $P$ contains $Q+\ep\ones$,  we have $\bot_Q + \ep\ones \in P$ and $f(\bot_Q + \ep\ones) = \bot_Q$. As $\bot_P \le \bot_Q + \ep\ones$, we have $f(\bot_P)\le \bot_Q $. Thus $f$ preserves the bottom, completing the proof.
\end{proof}

Let $f,g: X\to \R^d$ be two functions. We say that $g$ is \emph{$f$-compatible} (\cite{bjerkevik2021ell}) if,  for all $x,y\in X$ and $i\in [d]$,  $f(x)_i\le f(y)_i$ implies $g(x)_i\le g(y)_i$. This definition implies that, if $f(x)_i=f(y)_i$ for some $x,y\in X$ and $i\in [d]$, then $g(x)_i=g(y)_i$. 
\begin{lemma}\label{lem:compatibility}
     Let $P$ and $Q$ be the smallest grid containing $f(X)$ and $g(X)$, respectively.
If $g$ is $f$-compatible, then there is a grid morphism $h:P\to Q$ such that $h\circ f=g$ and $\dist{h}\le \norm{f-g}_\infty$.
\end{lemma}
\begin{proof}
    By \Cref{rem:smallest grid}, we have that $P=\prod_{i=1}^d \pi_i(f(X))$ and $Q=\prod_{i=1}^d \pi_i(g(X))$.
    In order to construct $h:P\to Q$, we construct, for each $i$, a monotone function $h_i:\pi_i(f(X))
\to \pi_i(g(X))$.
Fix $i\in [d]$ and let $p_i\in \pi_i(f(X))$. Then, there exists $x\in X$ such that $f(x)_i=p_i$. We define $h_i(p_i)$ as $g(x)_i\in \pi_i(g(x))$; this map $h_i$ is well-defined and monotone since $g$ is $f$-compatible. In addition, we have that \[\abs{p_i-q_i}= \abs{f(x)_i-g(x)_i} \le \norm{f-g}_\infty,\]
implying that $\norm{h_i}\le \norm{f-g}_\infty$ and in turn $\norm{h}\le \norm{f-g}_\infty$. %

For any $i\in [d]$ and any $x\in X$, we have $h_i(f(x)_i)=h_i(p_i)=q_i=g(x)_i$, and thus $h_i\circ (\pi_i \circ f) = \pi_i\circ g$, as functions $X\to \pi_i(g(X))$. As $h_i\circ \pi_i|_{P}=\pi_i|_{Q}\circ h$, we have that 
\[\pi_i \circ g = h_i\circ \pi_i \circ f = \pi_i \circ (h \circ f).\] Since this equation holds for every $i\in [d]$, we have that $g=h\circ f$, as desired.
\end{proof}

\begin{proposition}\label{editinterleaving}
   For any finitely presented $\R^d$-persistence modules $M$ and $N$,
    $\dedi^\carte(M,N)\le \dint(M,N)$. %
\end{proposition}

\begin{remark}
\label{rem:in-depth} 

  In proving \Cref{editinterleaving}, we utilize the idea of \emph{decomposing} a geodesic path between persistence modules with respect to the interleaving distance. This idea has appeared in the literature. For example, in the proofs of \cite[Theorem~1.7(iv)]{bjerkevik2021ell} and \cite[Theorem~7.1]{gulen2022galois}, geodesic paths are decomposed into pieces, each of which corresponds to our notion of Galois-edit, although this notion is not made explicit there. %

More broadly, the idea of decomposing a geodesic path trace back to earlier works, including \cite[Section~3.3]{cohen2007stability} and \cite[Section~4]{cohen2006vines}; however, these works are formulated
in terms of the bottleneck distance (rather than the interleaving distance) and do not involve the notion of Galois connection. 

In fact, we provide two different proofs of a critical claim in the proof of \Cref{editinterleaving}. The first proof parallels that of \cite[Theorem 1.7 (iv)]{bjerkevik2021ell} (and %
also resembles
that of \cite[Theorem 7.1]{gulen2022galois}), while the second proof utilizes the compactness of the interval $[0,1]\subset \R$ as in the proof of the bottleneck stability of the persistence diagram \cite[Section 3.3]{cohen2007stability}. 

\end{remark}

\begin{proof}[Proof of \Cref{editinterleaving}] %

    If $\dint(M,N)=\infty$, there is nothing to prove. Assume that $\ep=\dint(M,N)<\infty$. As the persistence modules $M$ and $N$ are finitely presented, by \cite[Theorem 6.1]{lesnick2015theory}, $M$ and $N$ are $\ep$-interleaved. 
    By \Cref{lesnickpresentation}, there exist $d$-graded sets $W_1,W_2$ and homogeneous sets $Y_1\subset \fr[W_1,W_2(-\ep)]$, $Y_2\subset \fr[W_1(-\ep),W_2]$ such that $M\cong \langle W_1,W_2(-\ep)|Y_1,Y_2(-\ep)\rangle$ and $N\cong \langle W_1(-\ep),W_2|Y_1(-\ep),Y_2\rangle$.

    Consider the set $X=\sym(W_1)\sqcup \sym(W_2) \sqcup \sym(Y_1) \sqcup \sym(Y_2)$. Denote by $\gr_M$ and $\gr_N$ the grading functions of the presentations of $M$ and $N$, respectively. %
    $$ F_t:= (1-t)\gr_M+t\gr_N \quad \quad t\in [0,1],$$
    such that $F_0=\gr_M$ and $F_1=\gr_N$. Then we have an associated presentation $M_t=\langle W_1(-\ep t),W_2(-\ep(-1+t)) | Y_1(-\ep t),Y_2(-\ep(1+t))\rangle $ such that the grading function of $M_t$ is exactly $F_t$. And of course, $M_0\cong M$ and $M_1\cong N$.

    We finish the proof by showing the following claim.
    
    \vspace{2mm}
    \noindent \emph{Claim.} There is a finite sequence $\{0\le t_0< t_1<\cdots t_n\le 1\}$ such that for each $1\le i<n$, there is a $\grid$-edit $f_i$ between $M_{t_i}$ and $M_{t_{i+1}}$ with $\dist{f_i}\le \dist{F_{t_{i+1}}-F_{t_i}}_\infty %
    $.

    The claim indeed completes the proof since 
    \[ \dedi^{\grid} (M,N)\le \sum_{i=1}^{n-1} \dist{f_i} \le \sum_{i=1}^{n-1} \norm{F_{t_{i+1}}-F_{t_i}}_\infty\sum_{i=1}^{n-1}\ep(t_{i+1}-t_i)=\ep(t_n-t_1)=\ep.\]

    As mentioned before, we provide two different proofs of the claim. %
    \vspace{2mm}

   \noindent \emph{First Proof of \emph{Claim}:}    
    We say $t\in[0,1]$ is \emph{critical}, if there exist $x,y\in X$ and $i\in [d]$ such that $F_t(x)_i=F_t(y)_i$ and there is a $t'\in[0,1]$ such that $F_{t'}(x)_i\neq F_{t'}(y)_i$. In words, $t\in [0,1]$ is critical if grid lines merge at time $t$ along the continuous evolution of  $P$ into $Q$ given by $(F_t)_{t\in[0,1]}$. 
    Recall also the definition of injectivity radius before \Cref{lem:localstability}. As $X$ is finite, there are only finitely many critical points in $[0,1]$. Let $\{0\le t_0<t_1<\cdots <t_n\le 1\}$ be the set of critical points, together with $0$ and $1$. For convenience, let $t_0:=0$ and $t_{n}:=1$. For any $0\le j<n$, consider any $s\in (t_{j-1},t_j)$. To apply \Cref{lem:compatibility}, we verify the conditions which says:
    For any $x,x'\in X$ and $i\in [d]$, if $F_s(x)_i\le F_s(x')_i$, then $F_{t_j}(x)_i\le F_{t_j}(x')_i$ and $F_{t_{j-1}}(x)_i\le F_{t_{j-1}}(x')_i$. %
    As there is no critical point in $(t_{j-1},t_j)$, and the value of $F_t(x)_i$ is linear in $t$, either $F_{t_j}(x)_i>F_{t_j}(x')_i$ or $F_{t_{j-1}}(x)_i>F_{t_{j-1}}(x')_i$ implies $F_s(x)_i> F_s(x')_i$.
    Applying \Cref{lem:compatibility}, we have grid functions $h$ and $k$ such that $h\circ F_s=F_{t_j}$ and $k\circ F_s=F_{t_{j-1}}$. 
    So choose $t_\frac{2j-1}{2}$ for every $0\le j <n$, so that we have the following sequence of functions
    $$ F_{t_0}, F_{t_\frac12}, F_{t_1},\ldots ,F_{t_j},F_{t_{\frac{2j+1}{j}}},F_{t_{j+1}}, \ldots, F_{t_{n-1}}, F_{\frac{2n-1}{2}}, F_{t_{n}}.$$

    For any adjacent functions $F_t$ and $F_{t'}$ in the sequence above, %
    there is a grid morphism $h$ which preserves the bottom and satisfies either $h\circ F_t=F_{t'}$ or $h\circ F_{t'}=F_t$. Therefore, it gives a $\grid$-edit between $M_t$ and $M_{t'}$ by \Cref{cor:edit_by_monotone_function}.
    As stated in \Cref{lem:compatibility}, we have $\dist{h}\le \norm{F_t-F_{t'}}_\infty$. There is a $\grid$-edit between $M_t$ and $M_{t'}$ of cost $\le \norm{F_t-F_{t'}}_\infty$. %

    \vspace{2mm}
   \noindent \emph{Second Proof of \emph{Claim}:}   For $t\in [0,1]$, let $r(t):=\inj(F_t)/\ep$, which is positive. Then for any $s\in (t-r(t),t+r(t))$, we have $\norm{F_t-F_{s}}=\ep\abs{t-{s}}<\inj(F_t)$. By \Cref{lem:localstability}, there is a $\grid$-edit $f$ between $M_t$ and $M_{s}$ such that $\dist{f}\le \norm{F_t-F_{s}}_\infty.$ %

    The open intervals $U_t:=(t-r(t),t+r(t))$ for $t\in[0,1]$ cover $[0,1]$. By compactness, there is a finite $S=:\{t_0<t_1<\cdots<t_n\}\subset [0,1]$ such that $U_t$ for $t\in S$ cover $[0,1]$. 
    We assume $U_{t_i}\cap U_{t_{i+1}}$ is nonempty for $0\le i<n$ by choosing $S$ to be minimal. Also, we may assume that $t_0=0$ and $t_n=1$, and still we have the property that $U_{t_i}\cap U_{t_{i+1}}$ is nonempty. Consider the sequence $0=t_0\le t_{\frac12}\le t_1\cdots\le t_{n-1}\le t_{\frac{2n-1}2}\le t_n=1$ by choosing $t_{\frac{2i+1}{2}}\in U_i\cap U_{i+1}$ for each $0\le i<n$. Then for any adjacent $t$ and $t'$ in the sequence, there is a $\grid$-edit between $M_t$ and $M_{t'}$ of cost $\le \norm{F_t-F_{t'}}_\infty$.%

\end{proof}

\begin{remark}\label{rem:existence_of_optimal_edit_sequence}
The first proof of %
the claim and the proof of \Cref{lem:compatibility} %
provide a construction of a $\carte$-edit sequence between any pair of finitely presented $\R^d$-persistence modules $M$ and $N$ with $\dint(M,N) <\infty$, whose cost is at most $\dint(M,N)$. Since $\carte$ is a subcategory of $\Poset$, this sequence is also a $\Poset$-edit sequence, implying that $\dedi^{\Poset}(M,N)\leq \dint(M,N)$. Invoking $\dint(M,N)\leq \dedi^{\Poset}(M,N)$ from \Cref{prop:interleavingedit}, we obtain that the cost of this sequence achieves the value of $\dedi^{\Poset}(M,N)$, i.e., such a $\grid$-edit sequence is %
optimal among all $\Poset$-edit sequences between $M$ and $N$.%
\end{remark}

\begin{proof}[Proof of \Cref{thm:dedi=dint}]
    By \Cref{prop:interleavingedit}, \Cref{relationbetweenedits}~\ref{item:relationbetweenedits1}, and \Cref{editinterleaving}, we have that
    \(\dint(M,N)\le \dedi^\Poset(M,N)\le  \dedi^\jsLat(M,N) \le \dedi^{\carte}(M,N)\le \dint(M,N).\)
\end{proof}

\section{Alternative proof of the bottleneck stability theorem}
\label{sec:functoriality}

In this section, we show that any $\mathsf{Grid}_1$-\edit{} $f:M\to N$ induces a matching of cost at most $\dist{f}$ between $\barc(M)$ and $\barc(N)$ (\Cref{thm:functoriality_of_barcode_1} and \Cref{cor:bottleneck_le_dist}).
Invoking Theorem~\ref{thm:dedi=dint} and \Cref{cor:single-edit is enough}, these provide an alternative proof of the well-known  bottleneck stability theorem (\Cref{thm:bottleneck stability}).

\begin{theorem}[Galois-Induced Matching Theorem]\label{thm:functoriality_of_barcode_1}
    For any $\mathsf{Grid}_1$-\edit{} $f:M\to N$, we have that 
    \begin{equation}\label{barcode calculation formula}
        \barc(N)=f\ \barc(M):=\lmulti [f(p),f(q)):[p,q)\in \barc(M),\ f(p)<f(q)\rmulti,
    \end{equation}
    with the convention that $f(\infty):=\infty$. 
\end{theorem}

We provide two different proofs of this theorem: The first proof utilizes the Galois-Induced Presentation Theorem (\Cref{thm:Galois_edit_is_presentation_edit}) and the second proof utilizes the M\"obius inversion theorem.%

\begin{proof}[First Proof of \Cref{thm:functoriality_of_barcode_1}]\label{proof:barcde_calculation_formula_using_presentation}
   Let $\barc(M)=:\lmulti [b_i,d_i):i\in [n]\rmulti$ for some $n\in \N$.
    Then we have $M\cong\langle W|Y \rangle$, where $W=\{(i,b_i):i\in [n]\}$ and $Y=\{(i,d_i):i\in [n] \text{ and } d_i\neq \infty\}$. By \Cref{thm:Galois_edit_is_presentation_edit}, we have $N\cong \langle f(W)|f(Y)\rangle$. Note that $f(W)=\{(i,f(b_i):i\in [n]\}$ and $f(Y)=\{(i,f(d_i):i\in [n] \text{ and } d_i\neq \infty\}$. Therefore, $\barc(N)= \lmulti [f(b_i),f(d_i)):i\in [n], f(b_i)<f(d_i) \rmulti$, as desired.
\end{proof}

For the second proof %
, we first review the relationship between the rank function and the barcode of a finitely presented $\R$-persistence module.

Let $P:=\{p_1<p_2<\cdots<p_n\}\subset\R$. 
Suppose that an $\R$-persistence module $M$ is $P$-presented. Then $M$ is constant on the half-open intervals $[p_i,p_{i+1})$ for $1\le i<n$, and $[p_n,\infty)$. 
    The \textbf{rank function} of $M$ is 
        $\rk(M):\Dgm(\R)\to \Nz$ given by, for each $[a,b)\in \Int(\R)$,
        \begin{equation}
        \rk(M)([a,b))=\begin{cases}
            \rank \ M_{a\le b^-}
            &\text{if $b=p_2,\ldots,p_n,\infty$}\\
            \rank \ M_{a\le b} &
            \text{otherwise,}
        \end{cases}
    \end{equation}    
   where $b^-$ is the largest element of $P$ strictly less than $b$.
The M\"obius inversion theorem \cite{rota1964foundations} directly implies: 
\begin{theorem}[Algebraic-combinatorial view on the barcode {\cite{abeasis1981geometry,patel2018generalized}}]\label{thm:mobius} 
The unique function $f:\Int(\R)\to \R$ satisfying the condition 
\[\rk(M)(I)=\sum_{\substack{J\in \Int(\R)\\J\supset I}}f(J), \ \ \mbox{for all } I\in\Int(\R)\]
is given by $f:J\mapsto \mbox{(the multiplicity of $J$ in $\barc(M)$)}$.
\end{theorem}

 \begin{proof}[Second Proof of \Cref{thm:functoriality_of_barcode_1}]

  Suppose that $P,Q\subset \R$ are finite, $M$ is $P$-presented, $N$ is $Q$-presented, $f:P\to Q$ is a grid morphism that has a right adjoint $g$, and $M|_P\circ g \cong N|_Q$. We have that $f(\bot_P)= \bot_Q$ by \Cref{productofadjoints}~\ref{item:productofadjoints1}. 
  
Since $M$ is $P$-presented, $\barc(M)$ consists solely of intervals whose endpoints are in $P\cup\{\infty\}$.
By $\Int(\R;P)$, let us denote the set of all such intervals. In the following proof, $\barc(M)$ will be regarded as the function $\Int(\R)\to \Nz$
given by $J\mapsto \mbox{(the multiplicity of $J$ in $\barc(M)$)}$ whose support is contained in $\Int(\R;P)$. A similar convention is assumed for $f\ \barc(M)$ and for $\barc(N)$.

Let $p\in P\cup \{\infty\}$ be an element other than the minimum. Let $p^-$ denote the \emph{predecessor} of $p$, i.e., the maximum of the subposet $p\bindownarrow\setminus\{p\}\subset P$.  For any $q\in P\cup \{\infty\}$, we have that $p^-<q$ if and only if $p\le q$. Thus for any $[p,q), [s,t)\in \Int(\R;P)$, $[p,q)\subset [s,t)$ if and only if $[p,q^-]\subset [s,t)$. 
The following claim is useful.\vspace{2mm}
    
\noindent    \emph{Claim.} For $s,t\in Q\subset \R$ with $s\leq t$ and $[p,q)\in \Int(\R;P)$, 
\begin{equation}\label{eq:equivalence}
[s,t]\subset [f(p),f(q)) \iff [g(s),g(t)]  \subset [p,q).
\end{equation}
To prove the claim,
        it suffices to show that
        \begin{enumerate}[label=(\roman*)]
            \item $f(p)\le s \iff p \le g(s)$, and \label{item:first}
            \item $t< f(q)\iff g(t)<q$. 
            \label{item:second}
        \end{enumerate} \Cref{{item:first}} is immediate from the assumption that $f\dashv g$. Let us prove \Cref{item:second}. When $q=\infty$, there is nothing to show and thus assume that $q<\infty$. Then, 
     \[   t< f(q) \iff f(q)\not\le t
        \iff q \not\le g(t)
        \iff g(t)< q,\]
    completing the proof of the claim.
    
    For any $[s,t)\in \Int(\R;Q)$, we have
    \begin{align*}
        \sum_{\substack{[u,v)\in \Int(\R;Q)\\
    [s,t)\subset[u,v)}} f\ \barc(M)([u,v))
    &=\sum_{\substack{[u,v)\in \Int(\R;Q)\\
    [s,t)\subset[u,v)}}\sum_{\substack{[p,q)\in \Int(\R;P)\\
    [u,v)=[f(p),f(q))}} \brac(M)([p,q))\\
    &=\sum_{\substack{[p,q)\in \Int(\R;P)\\
    [s,t^-]\subset[f(p),f(q))}}\brac(M)([p,q))\\
    &=\sum_{\substack{[p,q)\in \Int(\R;P)\\
    [g(s),g(t^-)]\subset[p,q)}}\brac(M)([p,q))&\mbox{by \emph{Claim}}\\
    &=\rk \ M_{g(s)\le g(t^-)}\\
    &=\rk \ N_{s\le t^-}&\mbox{since  $M|_P\circ g \cong N|_Q$}.
    \end{align*}
By Theorem~\ref{thm:mobius}, we also have that \[\rk \ N_{s\le t^-}=\sum_{\substack{[u,v)\in \Int(\R)\\
    [s,t)\subset[u,v)}}\barc(N)([u,v))=\sum_{\substack{[u,v)\in \Int(\R;Q)\\
    [s,t)\subset[u,v)}}\barc(N)([u,v)).\] By the uniqueness stated in that theorem, we obtain $\barc(N)=f\ \barc(M)$, as desired.
    \end{proof}

\begin{remark} 
One may wish to utilize the Galois-edit formulation of the interleaving distance (\Cref{thm:dedi=dint}) to seek alternative proofs for stability of other invariants of multi-parameter persistence. A notable candidate is the well-known stability of the fibered barcode under the matching distance \cite{cerri2013betti, landi2018rank}, with a recent alternative proof \cite[Theorem 1.7 (iii)]{bjerkevik2021ell}.

However, \Cref{thm:dedi=dint} does not yield a simple alternative proof of the stability of the fibered barcode; it is not diffcult to observe that a Galois-edit between persistence modules $M \to N$ does not necessarily induce a Galois-edit $M|_\ell \to N|_\ell$ between their restrictions to a line $\ell\subset \R^d$, but instead a Galois-edit sequence.%

\end{remark}

   In \Cref{appendix:optimal sequence can be long}, we review the notions of matchings between barcodes and the bottleneck distance. From \eqref{barcode calculation formula}, it is clear that $f$ induces a matching of cost at most $\dist{f}$ between $\barc(M)$ and $\barc(N)$, i.e., $\bott\circ(\barc\times\barc)$ is 1-Lipschitz under a single $\calE$-edit:

\begin{corollary}\label{cor:bottleneck_le_dist}
    For any $\mathsf{Grid}_1$-\edit{} $f:M\to N$, 
    $\bott(\barc(M),\barc(N))\le \dist{f}.$
\end{corollary}

Since $\dint=\dedi^\calE$ (\Cref{thm:dedi=dint}), %
the preceding corollary and \Cref{cor:single-edit is enough}  directly imply: 
\begin{theorem}[Bottleneck stability theorem for finitely presented persistence modules {\cite{chazal2009proximity,chazal2016structure}}]\label{thm:bottleneck stability}
    For any finitely presented $\R$-persistence modules $M$ and $N$, 
     $\bott(\barc(M),\barc(N))\le \dint(M,N).$
\end{theorem}

In the following remark, we highlight how the Galois connection and the Galois-edit reformulation of the interleaving distance (Theorem~\ref{thm:dedi=dint}) facilitate the proof above.

\begin{remark}\label{rem:usefulness of Galois connection}
We exploit the \emph{Galois connection} $f \dashv g$ to establish the equivalence given in \Cref{eq:equivalence}; this equivalence facilitates the application of Theorem~\ref{thm:mobius}, which is an application of the \emph{M\"obius inversion theorem}. In addition, Theorem~\ref{thm:dedi=dint} allows us to utilize  the \emph{maximality of the edit distance} (\Cref{prop:edit_is_universal}).
\end{remark}

\section{Discussion}\label{sec:discussion}

Via a notion of Galois-edit between persistence modules, we establish a Galois-edit distance formulation of the interleaving distance (\Cref{thm:dedi=dint}). A key step for this theorem is to establish a correspondence between the notion of Galois-edit for persistence modules and, in a suitable sense, that for their presentations (\Cref{thm:Galois_edit_is_presentation_edit}).

As a consequence of this Galois-edit formulation of the interleaving distance, we obtain characterizations of interleaving-stable invariants in terms of edit-stable invariants, and vice versa, for multi-parameter persistence modules (\Cref{cor:one stability implies another,cor:single-edit is enough}). To demonstrate the utility of these results, we provide a succinct proof of the bottleneck stability theorem (\Cref{thm:bottleneck stability}). We expect our results to expedite the exploration of stable invariants in multi-parameter persistence.

\bibliographystyle{plain}
\bibliography{bib.bib}

\appendix
\crefalias{section}{appendix}
\addcontentsline{toc}{section}{Appendix}

\addtocontents{toc}{\protect\setcounter{tocdepth}{0}}

\section{Omitted proofs}\label{appendix:omitted proofs}

\begin{proof}[Proof of \Cref{prop:components_of_category}] %

We assume that $\calE=\jsLat$ or $\grid$; the remaining case $\calE=\Poset$ follows from \Cref{thm:dedi=dint}.

First, we prove the forward implication.
Let $f:M\to N$ be an $\calE$-\edit{}, i.e., $M$ is $P$-presented, $N$ is $Q$-presented,  $f:P\leftrightarrows Q:g$ is a Galois connection, and $M|_P\circ g=N|_Q$. 
By \Cref{productofadjoints}~\ref{productofadjoints0}, the right adjoint $g$ preserves the top, and thus $N(\top_Q)\cong M(g(\top_Q))=M(\top_P)$, and $N(\infty)\cong M(\infty)$. 
Therefore, if $M$ and $N$ are $\R^d$-persistence modules with $M(\infty)\not\cong N(\infty)$, then there is no $\calE$-edit sequence between $M$ and $N$, or equivalently $\dedi^\calE(M,N)=\infty$.
    
Next, we prove the backward implication.
Let $M$ be any $P$-presented $\R^d$-persistence module for some $P\in \mathrm{ob}(\calE)$. Let $f: P\to \{\vec{0}\}$ be the unique monotone function. By \Cref{prop:existence of adjoint},
$f$ has the right adjoint $g:\{\vec{0}\}\to P$ defined by $g(\vec{0})=\top_P$. Then $M|_P\circ g$ is the persistence module on the singleton $\{\vec{0}\}$ given by $\vec{0}\mapsto M(\infty)$. 
Thus, for the free $\R^d$-persistence module $\left(\F_{\vec{0}\binuparrow}\right)^n:=\bigoplus_{0\leq i <n} \F_{\vec{0}\binuparrow}$ with $n=\dim M(\infty)$, we have that
\[M|_P\circ g\cong \left(\F_{\vec{0}\binuparrow}\right)^n\Big|_{\{\vec{0}\}}.\]
This shows that $f$ is an $\calE$-\edit{} from $M$ to $\left(\F_{\vec{0}\binuparrow}\right)^n$ with $\dist{f}<\infty$. 
Thus, for any $\R^d$-persistence module $N$ with $M(\infty)\cong N(\infty)$,
we have $\dedi^{\calE}(M,N)\le \dedi^{\calE}\left(M,\left(\F_{\vec{0}\binuparrow}\right)^n\right)+\dedi^{\calE}\left(\left(\F_{\vec{0}\binuparrow}\right)^n,N\right)  <\infty$.
\end{proof}

\begin{proof}[Proof of \Cref{lem:edit_induces_lesnick_interleaving'}]
    Let $\ep:=\norm{f}$.
  We wish to define natural transformations \[F:\langle W_1 | Y_1\rangle \Rightarrow \langle f(W_1) | f(Y_1)\rangle(\ep)\quad \text{ and }\quad G:\langle f(W_1) | f(Y_1)\rangle\Rightarrow \langle W_1 | Y_1\rangle(\ep)\] such that Condition \eqref{eq:interleaving} holds.  Fix $a\in \R^d$. 
  Recall that for any graded set $W$, $W^{a\bindownarrow}:=\{w\in W: \gr(w)\le a\}$.
   Since $\ep=\norm{f}$,
   \begin{enumerate}[label=(\roman*)]
   \item  $\sym(W_1^{a\bindownarrow})\subset \sym\bigl(f(W_1)^{(a+\vec\ep)\bindownarrow}\bigr)$ and $\sym\bigl(f(W_1)^{a\bindownarrow}\bigr)\subset \sym(W_1^{(a+\vec\ep)\bindownarrow})$, and \item  $\sym(Y_1^{a\bindownarrow}) \subset \sym\bigl(f(Y_1)^{(a+\vec\ep)\bindownarrow}\bigr)$ and $\sym\bigl(f(Y_1)^{a\bindownarrow}\bigr) \subset \sym(Y_1^{(a+\vec\ep)\bindownarrow})$.
   \end{enumerate}
   
    Then, we define $F_a:\langle W_1 | Y_1\rangle_a \to \langle f(W_1) | f(Y_1)\rangle_{a+\vec\ep}$ as the composition of the canonical maps
    \[
    \langle W_1 | Y_1\rangle_a=\langle \sym(W_1^{a\bindownarrow}) | \sym(Y_1^{a\bindownarrow})\rangle_\F \to\langle \sym\bigl(f(W_1)^{(a+\vec\ep)\bindownarrow}\bigr) | \sym\bigl(f(Y_1)^{(a+\vec\ep)\bindownarrow}\bigr)\rangle_\F=\langle f(W_1) | f(Y_1)\rangle_{a+\vec\ep}.
    \]
    Similarly, define $G_a:\langle f(W_1) | f(Y_1)\rangle_a \to \langle W_1 | Y_1\rangle_{a+\vec\ep}$ as the map induced by the inclusion $\sym\bigl(f(W_1)^{a\bindownarrow}\bigr)\hookrightarrow \sym(W_1^{(a+\vec\ep)\bindownarrow})$,
    which is well-defined since $\sym\bigl(f(Y_1)^{a\bindownarrow}\bigr) \subset \sym(Y_1^{(a+\vec\ep)\bindownarrow})$.  
    
    The composition $G_{a+\vec\ep}\circ F_a$ is the map 
    \[
    \langle W_1 | Y_1\rangle_a=\langle \sym(W_1^{a\bindownarrow}) | \sym(Y_1^{a\bindownarrow})\rangle_\F \to \langle \sym(W_1^{(a+2\vec\ep)\bindownarrow}) | \sym(Y_1^{(a+2\vec\ep)\bindownarrow})\rangle_\F=\langle W_1 | Y_1\rangle_{a+2\vec\ep}
    \] 
    which sends, for each $w\in W_1$,  $\overline{w}\in \langle W_1 | Y_1\rangle_a$ to $\overline{w}\in \langle W_1 | Y_1\rangle_{a+2\vec\ep}$. This map is identical to the structure map 
    $ \langle W_1 | Y_1\rangle_a \to \langle W_1 | Y_1\rangle_{a+2\vec\ep}$ of $\langle W_1 | Y_1\rangle$ as described in \Cref{rem:vector space of modules}. By symmetry, we also have that $F_{a+\vec\ep}\circ G_a$ equals the structure map   $ \langle f(W_1) | f(Y_1)\rangle_a \to \langle f(W_1) | f(Y_1)\rangle_{a+2\vec\ep}$ of $\langle f(W_1) | f(Y_1)\rangle$. 
\end{proof}

\section{Weaker statement than \Cref{prop:interleavingedit} with an order-theoretic proof
}
\label{appendix:another proof}

In this section, we provide a proof of a weaker statement than \Cref{prop:interleavingedit} 
that does not use presentations of persistence modules, and is instead more order-theoretic, adapting ideas from \cite[Proposition 6.4]{gulen2022galois}.
To this end, we begin with a reinterpretation of \Cref{def:interleaving}.
 \begin{remark}[{\cite[Section 1.4]{bubenik2017interleaving}}]\label{rem:interleaving}  %
 For any $\ep\geq0$, the poset $\R^d \sqcup_\ep \R^d$ is defined as the set \[\R^d\sqcup\R^d=\{(a,t):a\in \R^d,\ t=1,2\},\] equipped with the order $\le_\ep$ given by $(a,s)\le_\ep (b,t)$ if and only if $a+\abs{s-t}\vec\ep \le b$ for $s,t\in \{1,2\}$. We denote the inclusions $a\mapsto (a,1)$ and $b\mapsto (b,2)$ by $\dotlessi_1$ and $\dotlessi_2$, respectively.

 Any two $\R^d$-persistence modules $M$ and $N$ are $\ep$-interleaved if and only of there exists an $\ep$-\emph{interleaving} between $M$ and $N$, i.e., a persistence module $L:\R^d \sqcup_\ep\R^d\to \vect$ such that the following diagram commutes.
    $$\begin{tikzcd}
        \R^d\arrow{r}{\dotlessi_1} \arrow[ddr, "M"']& \R^d\sqcup_\ep\R^d \arrow[dd, "L",dotted] &  \R^d\arrow[ddl, " N"]\arrow[l, swap, "\dotlessi_2"]\\
         &&\\
         &\vect &
    \end{tikzcd}$$
\end{remark}

\begin{proposition}[Weaker version of {\Cref{prop:interleavingedit}}] 
    \label{prop:interleavingedit_weak}
    For any finitely presented $\R^d$-persistence modules $M$ and $N$, we have that
    $ d_I(M,N)\le \dedi^{\JsLat}(M,N).$ 
\end{proposition}

\begin{proof}

By assumption, there exist subposets $P,Q\subset \R^d$ such that $M$ is $P$-presented and $N$ is $Q$-presented. By \Cref{prop:edit_is_universal}, it suffices to show the following claim.\vspace{3mm}

   \noindent\emph{Claim.}
   Let $f:P\to Q$ be a monotone function with a right adjoint $g:Q\to P$ such that $M|_P\circ g=N|_Q$, i.e.,  $f$ is a $\jsLat$-\edit{} $M\rightarrow N$. Then, $\dint(M,N)\leq \dist{f}$.

    We prove this claim. Let $i_1$ and $i_2$ be the inclusions $P\hookrightarrow \R^d$ and $Q\hookrightarrow \R^d$, respectively. By     \Cref{productofadjoints}~\ref{item:productofadjoints2}
    , the induced maps $(i_1)_\bot:P_\bot\hookrightarrow \R^d_\bot$ and $(i_2)_\bot:Q_\bot\hookrightarrow \R^d_\bot$ have right adjoints $\lfloor-\rfloor_P$ and $\lfloor-\rfloor_Q$. Then, the solid arrows in the following diagram commute. 
$$\begin{tikzcd}
    \R^d_\bot
      \arrow[ddr, "M_\bot"', bend right=20] 
      \arrow[r, bend left=15, "\lfloor-\rfloor_P"]     
    & P_\bot \arrow[dd, "(M|_P)_\bot"' {xshift = 2pt, yshift=0pt}]  \arrow[l, dashed,bend left=15, "(i_1)_\bot"]  \arrow[r, bend left=15, dashed, "f_\bot"] 
    & Q_\bot 
      \arrow[l, bend left=15, "g_\bot"] 
      \arrow[ddl, "(N|_Q)_\bot" {yshift=5pt, xshift=3pt}]
    \arrow[r, bend left=15, dashed, "(i_2)_\bot"] 
    & \R^d_\bot 
      \arrow[l, bend left=15,"\lfloor-\rfloor_Q"] 
      \arrow[ddll, bend left=30,"N_\bot" {yshift=4pt, xshift=4pt}] 
    \\
    &&& \\
    & \vect & &
\end{tikzcd}
$$
    Let $f_2:=(i_2)_\bot\circ f_\bot$ and $g_2:=g_\bot\circ \lfloor-\rfloor_Q$. Then, by \Cref{remark:galois basics}~\ref{rem:compositionofGalois}, we have that $f_2\dashv g_2$. Also, let $f_1:=(i_1)_\bot$ and $g_1:=\lfloor - \rfloor_P$ so that $f_1\dashv g_1$. 
    Then the solid arrows in the following diagram commute.
    $$\begin{tikzcd}
        \R^d_\bot \arrow[r,bend left=15, "g_1"]\arrow[ddr, "M_\bot"', bend right=20]& P_\bot \arrow[dd, "(M|_P)_\bot"' {xshift = 2pt, yshift = 0pt}]  \arrow[l, dashed,bend left=15, "f_1"]  \arrow[r, bend left=15, dashed, "f_2"]&  \arrow[l,bend left = 15, "g_2"] \R^d_\bot\arrow[ddl, "N_\bot", bend left=20]\\
         &&\\
         &\vect &
    \end{tikzcd}$$
    
   Define $Q'$ as the preordered set $(\R^d\sqcup \R^d,\preceq)$ where $(a,t)\preceq (b,s)$  iff $g_t(a)\le g_s(b)$ in $P_\bot$ for $a,b\in \R^d$ and $s,t\in \{1,2\}$. Let $L':Q'\to \vect$ be the functor defined by \[L'(a,t):=M_\bot(g_t(a)), \quad \quad L'((a,t)\preceq (b,s))={(M_\bot)}_{g_t(a)\le g_s(b)}.\] 
   Then, invoking that $M|_P\circ g=N|_Q$, the canonical inclusions $j_t:\R^d\to \R^d\sqcup \R^d=Q'$ for $t=1,2$ make the following diagram commute.
    \begin{equation}\begin{tikzcd}
        \R^d\arrow{r}{j_1} \arrow[ddr, " M"']& Q' \arrow[dd, "L'"] &  \R^d\arrow[ddl, " N"]\arrow[l, "j_2"']\\
         &&\\
         &\vect &
    \end{tikzcd}\label{eq:3}\end{equation}

    Let $\ep:=\dist{f}$. Recall the poset $\R^d\sqcup_\ep \R^d$ from \Cref{rem:interleaving}.\vspace{2mm}

    \noindent\emph{Subclaim.} The canonical bijection from $\R^d\sqcup_\ep\R^d=(\R^d\sqcup \R^d, \le_\ep)$ to $Q'=(\R^d\sqcup \R^d,\preceq)$ is monotone.

    \begin{proof} 
    Note that $\dist{f_2|_{P}}=\dist{i_2\circ f}=\ep$ as $i_2$ is a metric embedding. For $a,b\in \R^d$ and $t\in \{1,2\}$, if $(a,t)\le_\ep (b,t)$, then $a\le b$ by definition, and thus we have that $(a,t)\preceq (b,t)$. Therefore, it only remains to prove that if $(a,s)\le_\ep (b,t)$ with $s\neq t$, then $(a,s)\preceq (b,t)$.
    
    First, we show that if $(a,1)\le_\ep (b,2)$ (i.e.,  $a+\vec\ep \le b$), then $(a,1)\preceq (b,2)$. For this, it suffices to show that $(a,1)\preceq (a+\vec\ep,2)$, or equivalently, $g_1(a)\le g_2(a+\vec\ep)$ for any $a\in \R^d$. Fix $a \in \R^d$.
    If $g_1(a)=\bot$, then there is nothing to prove, and thus we assume that $g_1(a)\neq \bot$, which implies that $g_1(a)\in P\subset \R^d$.
    Since $f_2 \dashv g_2$ and $f_2(p)\leq f_2(p)$ for all $p\in P_\bot$, we have that $p\le g_2(f_2(p)) $ for all $p\in P_\bot$, and consequently \[g_1(a)\le g_2(f_2(g_1(a))).\] %
        Also, we have:
    \begin{align*}
        f_2(g_1(a))&\le g_1(a)+\vec\ep 
        && \text{since $g_1(a)\in P$ and $\dist{f_2|_{P}}=\ep$}\\
        &=f_1(g_1(a)) + \vec\ep && \text{since $f_1|_P$ is the inclusion $i_1:P\hookrightarrow\R^d$}\\
        &\le a +\vec\ep. && \text{since $g_1(a)\le g_1(a)$ and $f_1\dashv g_1$}
    \end{align*}
     Thus, we have that $g_1(a)\le g_2(f_2(g_1(a)))\le g_2(a+\vec\ep)$, as desired.

    Next, we show that if $(b,2)\le_\ep (a,1)$ (i.e.,  $b+\vec\ep \le a$), then $(b,2)\preceq (a,1)$. It suffices to show that $(b,2)\preceq (b+\vec\ep,1)$, or equivalently, $g_2(b)\le g_1(b+\vec\ep)$ for any $b\in \R^d$. Fix $b\in \R^d$. 
    If $g_2(b)=\bot$, then there is nothing to prove, and thus we assume that $g_2(b)\not=\bot$, i.e.,  $g_2(b)\in P\subset \R^d$.
    Since $f_1 \dashv g_1$ and $f_1(p)\leq f_1(p)$ for all $p\in P_\bot$, we have that $p\le g_1(f_1(p))$ for all $p\in P_\bot$, and consequently \[g_2(b)\le g_1(f_1(g_2(b))).\] 
    Also, we have:
    \begin{align*}
        f_1(g_2(b))&=g_2(b) && \text{since $f_1|_P$ is the inclusion $i_1:P\hookrightarrow\R^d$}\\
        &=g_2(b)-\overrightarrow{\dist{f_2|_P}}+\vec\ep&&\text {since $\dist{f_2|_P}=\ep$}\\
        &\le f_2(g_2(b))+\vec\ep&& \text{since $g_2(b)\in P$ and by the definition of $\dist-$}\\
        &\le b+\vec\ep && \text{since $g_2(b)\le g_2(b)$ and $f_2\dashv g_2$}\\
    \end{align*}
    Thus, we have that $g_2(b)\leq g_1(f_1(g_2(b)))\le g_1(b+\vec\ep)$, as desired.
    \end{proof} 

    By the previous subclaim, we obtain the following commutative diagram
    $$\begin{tikzcd} 
        \R^d\arrow{r}{\dotlessi_1} \arrow[ddr, "M"']& \R^d \sqcup_\ep \R^d \arrow[dd, " L"] &  \R^d\arrow[ddl, "N"]\arrow[l, "\dotlessi_2"']\\
         &&\\
         &\vect &
    \end{tikzcd}$$\\
    where $L$ is the restriction of $L'$ to $\R^d\sqcup_\ep \R^d$.
    By \Cref{rem:interleaving}, we have that $\dint(M,N)\le \ep=\dist{f}$, completing the proof of \emph{Claim}.
\end{proof}

\section{Comparison of constructibility and finite presentability}\label{appendix:two constructibility}

In this section, we show that, when restricted to $\R^d$-persistence modules, the notion of constructibility in \cite{gulen2022galois} is, in a suitable sense, equivalent to finite presentability %
given in our \Cref{KanExtensionOfPersistenceModules}~\ref{item:constructible}.

\begin{definition}[{\cite[Definition 2.10, Proposition 2.11, Definition 4.1]{gulen2022galois}}]\label{def:constructibility}
 Let $P$ be a poset.
\begin{enumerate}[label=(\roman*)]
    \item A \textbf{co-closure operator} on $P$ is a monotone function $c:P\to P$ such that $c$ is the right adjoint of the inclusion $c(P)\hookrightarrow P$. Equivalently, a monotone function $c:P\to P$ is a co-closure operator on $P$ if and only if $c(p)\le p$ for all $p\in P$ and $c\circ c= c$. 
    \item  \label{item:consructibility}A $P$-persistence module $M$ is called \textbf{constructible or $S$-constructible} if there is a finite set $S\subset P$ and there is a co-closure operator $c:P\to P$ with  image $S$ such that $M = M\circ c$. 
\end{enumerate}
\end{definition}

For any poset $P$ containing an infinite chain with no lower bound, no co-closure operator on $P$ can have a finite image. Consequently, for such $P$, there exists no constructible $P$-persistence module%
.
Since $\R^d$ is such a poset, any
$\R^d$-persistence module is not constructible%
. %
Nevertheless, the notions of constructibility and finitely presentability are equivalent only up to the inclusion of a bottom element:

\begin{proposition}\label{prop:constructibility_extends_original}
    Let $M$ be an $\R^d$-persistence module. Then, $M_\bot$ is constructible %
    if and only if  $M$ is finitely presented%
    .%
\end{proposition}

\begin{proof} %
    We first prove the forward implication. 
    Let $S\subset \R^d_\bot$ be  finite
    and assume that $M_\bot$ is $S$-constructible. Then,
    there is a co-closure operator $c: \R^d_\bot \to \R^d_\bot$ with its image $S$. The map $c$ is the right adjoint of the inclusion $S\hookrightarrow \R^d_\bot$. %
    Then, by \Cref{productofadjoints}~\ref{productofadjoints0}, $c$ %
    preserves all meets.  As $\R^d_\bot$ is a meet-semilattice and $c$ is surjective onto $S$, 
    $S$ also is a meet-semilattice. 
    We choose a sufficiently large $\ell\in \R$ so that $\vec \ell=(\ell,\ldots,\ell)$ is greater than all elements in $S$. 
    Then, since $c$ on $S$ is the identity and, $c$ is monotone,  $c(\vec\ell)$ is the top element in $S$.  
    This shows that $S$ is a finite bounded meet-semilattice; it is, therefore,  a lattice. 
    By \Cref{productofadjoints}~\ref{productofadjoints0}, the inclusion $S\hookrightarrow \R^d_\bot$ preserves all joins; in particular, $\bot_S=\bot_{\R^d_\bot}=:\bot$.
    As $S$ has all joins, for all $p,q\in S\setminus\{\bot\}$, we have $p \lor_{S} q = p \lor_{\R^d_{\bot}} q$. As the binary joins cannot be $\bot$, we have $p\lor_{S\setminus\{\bot\}} q = p \lor_{\R^d} q$. Thus, $S\setminus\{\bot\}$ is
    closed under binary joins in $\R^d$, i.e., $S\setminus\{\bot\}$ is a \joingrid{}. 
    
    By \Cref{productofadjoints}~\ref{item:productofadjoints2}, the right adjoint of $S\hookrightarrow \R^d_\bot$ is the floor function $\lfloor - \rfloor_{S\setminus\{\bot\}}$, i.e., $c=\lfloor - \rfloor_{S\setminus\{\bot\}}$. Therefore, $M_\bot = M_\bot \circ c =M_\bot \circ \lfloor - \rfloor_{S\setminus\{\bot\}}.$ By \Cref{item:Kan_extension_by_floor_function1}, $M$ is $P$-presented and thus finitely presented. %

    Next, we prove the reverse implication. Let $P\subset \R^d$ be a poset and suppose that $M$ is $P$-presented.%
    Let $Q$ be the smallest grid containing $P$ (cf.~\Cref{rem:smallest grid}). By \Cref{item:constructible over a larger poset}, $M$ is $Q$-presented. By \Cref{item:Kan_extension_by_floor_function1}, we have $M_\bot = M_\bot \circ \lfloor - \rfloor_Q$.
    By viewing $\lfloor -\rfloor_Q$ as a map $ \R^d_\bot\to \R^d_\bot$, it
    is a co-closure operator on $\R^d_\bot$ with finite image $Q_\bot$. %
    Thus, $M_\bot$ is constructible%
    . 
\end{proof}

\section{Comparison of the Galois-edit distance with earlier works}\label{appendix:comparison with gulen's work}
In this section, we compare the Galois-edit distance (\Cref{def:edit distance}) with other edit(-type) distances in the applied topology literature. We also compare the Galois-edit formulation of the interleaving distance (\Cref{thm:dedi=dint}) with previously known connections between the interleaving distance and Galois connections.

Most notably, \cite{gulen2022galois} establishes a direct connection between the interleaving distance and Galois connections. Recall from \cite{gulen2022galois} that a \textbf{Galois insertion} is a Galois connection $f : P\leftrightarrows{} Q : g$ satisfying $f \circ g = \mathrm{id}_Q$. 

\begin{definition}[{Rephrase of \cite[Definition~6.1]{gulen2022galois}}]\label{def:interleaving-galois}
Let $P$ and $Q$ be posets embedded in some metric space $(X, d_X)$. 
An $\ep$-interleaving between two persistence modules $M : P \to \vect$
and 
$N : Q \to \vect$ consists of a persistence module $\Gamma : R \to \vect$ 
with 
Galois insertions $f_M:R\leftrightarrows{}P:g_M$ and $f_N:R \leftrightarrows{} Q:g_N$ such that the solid arrows in the following diagram commute 
and $\ep=\max_{r \in R} d_X(f_M(r), f_N(r))$.
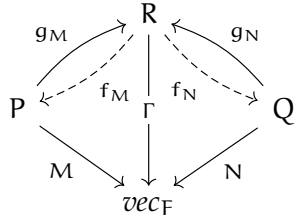
\[
\begin{tikzcd}
& R \arrow[dl, bend left=15, "f_M", dashed] \arrow[dr, bend right=15, "f_N"', dashed] \arrow[dd, "\Gamma" description] & \\
P \arrow[dr, "M"'] \arrow[ur, bend left=15, "g_M"] & 
& Q \arrow[dl, "N"] \arrow[ul, bend right=15, "g_N"'] \\
& \vect &
\end{tikzcd}
\]
The \textbf{interleaving distance} between $M$ and $N$ is the infimum of $\ep\geq0$ for which there exists an $\ep$-interleaving between $M$ and $N$.
\end{definition}

\begin{proposition}[{\cite[Proposition 6.4]{gulen2022galois}}]\label{prop:equivalence}Let $M$ and $N$ be 
$\R$-persistence modules, i.e., their indexing posets are $\R$. Then, the two interleaving distances between $M$ and $N$ from \Cref{def:interleaving,def:interleaving-galois} coincide.
\end{proposition}

We remark as follows on \Cref{def:interleaving-galois} and \Cref{prop:equivalence}.

\begin{enumerate}[label=(\roman*)]
\item 
\Cref{def:interleaving-galois} provides an edit-viewpoint on the interleaving distance between persistence modules in that the persistence module $\Gamma$, equipped with the two Galois insertions, is reminiscent of, in \cite{bauer2021reeb}, a pair of Reeb quotient maps from a Reeb graph, which corresponds to a single edit in a sequence of edits between Reeb graphs considered in \cite{bauer2021reeb}. 

\item 
In view of the previous item, the edit formulation of the interleaving distance presented in \Cref{def:interleaving-galois} uses only one- or two-step edits (depending on the perspective). In contrast, all of the edit distances in \cite{bauer2021reeb,mccleary2022edit} and our Galois-edit distance (\Cref{def:edit distance}) allow sequences of edits of arbitrary length. Furthermore, for our Galois-edit distance, an optimal edit sequence between two persistence modules can be arbitrarily long (cf. \Cref{prop:edit sequence of length two is not optimal}).%

\item Nevertheless, by taking a sequence of pullbacks from a Galois-edit sequence, one can obtain a one- or two-step edit sequence as considered in \Cref{def:interleaving-galois}. However, this sequence is not necessarily a Galois-edit sequence, since the pullback posets need not be embedded in the same ambient poset $\R^d$ as the indexing posets appearing in the original Galois-edit sequence. As seen in the proof of \Cref{thm:bottleneck stability}, having the same ambient dimension $d$ can play a crucial role in establishing stability results with respect to the Galois-edit distance (equivalently, the interleaving distance).

\item \Cref{prop:equivalence} establishes an isometry only for $\R$-persistence modules, while our \Cref{thm:dedi=dint} establishes an isometry for finitely presented $\R^d$-persistence modules for \emph{any} $d$. We remark that adapting the proof of \Cref{prop:equivalence} from \cite{gulen2022galois} to the multi-parameter setting is not straightforward.
\end{enumerate}

\section{Optimal edit sequence can be arbitrarily long}\label{appendix:optimal sequence can be long}

As noted in \Cref{rem:existence_of_optimal_edit_sequence}, for any $\R^d$-persistence modules $M$ and $N$, there exists an  $\calE$-edit sequence such that its cost attains $\dedi^\calE(M,N)$. In this section, we show that such a sequence can be arbitrarily long, by  proving:

\begin{figure}
\centering
\begin{tikzpicture}[xscale=2, yscale=2.2]

    \def\n{3}
    \def\ell{1}
    \def\ss{0.25}

    \def\xM{0.4}
    \def\xN{2}

    \tikzset{
        Mbar/.style={line width=3pt, line cap=round, blue!70},
        Nbar/.style={line width=3pt, line cap=round, red!70},
        axis/.style={thin, black!100},
        tick/.style={very thin, black!100}
    }

    \draw[axis,->] (0, 5/3) -- (0, 5.6);

    \foreach \k\l in {-1/{5/3},0/2,1/{7/3},3/3,6/4,9/5}{
        \draw[tick] (-0.12, \l) -- (0.12, \l);
        \node[left, black!100] at (-0.18, \l) {\normalsize \k};
    }

    \foreach \y/\yy/\r in {16/9/ \xN+\ss, 17/9 / \xN + 2*\ss, 2/1/\xN+3*\ss, 19/9/\xM+\ss, 20/9/\xM+2*\ss, 3/1/\xN+\ss, 4/1/\xN+2*\ss, 5/1/\xN+3*\ss}{
        \draw[gray!60, dashed] (0, \y/\yy) -- (\r, \y/\yy);
    }

    \foreach \j in {1,...,\n}{
        \ifnum\j<\n
            \pgfmathsetmacro{\birthM}{\j/9 + 2}
        \else
            \pgfmathsetmacro{\birthM}{2}
        \fi
        \pgfmathsetmacro{\deathM}{\j + 2}
        \draw[Mbar] (\xM+\j*\ss, \birthM) -- (\xM+\j*\ss, \deathM);
    }

    \foreach \j in {1,...,\n}{
        \pgfmathsetmacro{\birthN}{(\j-3)/9 + 2}
        \pgfmathsetmacro{\deathN}{\j + 2}
        \draw[Nbar] (\xN+\j*\ss, \birthN) -- (\xN+\j*\ss, \deathN);
    }

    \node[blue!70, font=\normalsize] at (\xM, 4.5) {$M$};
    \node[red!70, font=\normalsize] at (\xN, 4.5) {$N$};

\end{tikzpicture}
\caption{Barcodes of the persistence modules $M$ and $N$ given in the proof of \Cref{prop:edit sequence of length two is not optimal}.
}\label{fig:edit_cannot_be_shorter_than_3}
\end{figure}

\begin{proposition}\label{prop:edit sequence of length two is not optimal}
    For every $n\in \N$, 
    there exist $\R$-persistence modules $M$ and $N$ such that there is no  $\mathsf{Grid}_1$-edit sequence between $M$ and $N$ of length $n$ that attains the value of $\dedi^{\mathsf{Grid}_1}$.
\end{proposition}

Before proving the proposition, we review 
the notion of the bottleneck distance \cite{bauer2013induced,cohen2007stability}.

Let $\Int(\R):=\{[a,b):a<b\in \R\cup\{\infty\}\}$ be the poset ordered by reverse inclusion $\supset$.  
By abuse of notation, we regard an element $[a,b)$ in $\Int(\R)$ as a point in the extended plane $\R \times (\R\cup \{\infty\})$
and thus we write
\([a,b)\in \R \times (\R\cup \{\infty\}).\)
We define 
an extended metric on the extended plane by 
\[d_\infty([a,b),[c,d)):=\max\{\abs{c-a},\abs{d-b}\},\]
with the conventions $\infty-\infty=0$, $\infty-r=\infty$, $r-\infty=-\infty$ for $r\in\R$, and $\lvert-\infty\rvert=\infty$.

A \textbf{matching} from a set $S$ to a set $T$ (written as $\chi : S \not\to T$)
is a bijection $\chi : S' \to T'$, for some subsets $S' \subseteq S$ and $T' \subseteq T$.
Formally, we regard $\chi$ as a relation $\chi \subseteq S \times T$, where
$(s,t) \in \chi$ if and only if $s\in S'$ is sent to $t\in T'$ via $\chi$. Any $x\in S'\cup T'$ is called \textbf{matched}. Any $x\in S\cup T$ outside $S'\cup T'$ is called \textbf{unmatched}.
These notions extend naturally to multisets $S$ and $T$;
for a rigorous treatment, 
see \cite[Section 2]{bauer2013induced}.

We will call a finite multiset consisting of elements from $\Int(\R)$ a \textbf{barcode}.    For $I=[a,b)\in \Int(\R)$, let $c(I)=b-a$. For any matching $\chi$ between two barcodes $A$ and $B$,   let \[
    \mathrm{cost}(\chi):=\max\left\{\max\{ d_\infty (I,J):{(I,J)\in \chi}\},\ \frac{1}{2}\max \{c(I) :I\in A\cup B \mbox{ and $I$ is unmatched}\} \right\},
    \]
    where $\frac{1}{2}\infty:=\infty$.  
    The \textbf{bottleneck distance } between $A$ and $B$ is 
    $  \bott(A,B):=\min_{\chi:A\not\to B} \cost(\chi). 
    $

\begin{proof}[Proof of \Cref{prop:edit sequence of length two is not optimal}]
We only prove the case of $n=2$, which naturally extends to the case of arbitrary $n>2$.
    Assume that $M$ and $N$ have the following barcodes (see \Cref{fig:edit_cannot_be_shorter_than_3}).
    \begin{align*}
        \barc(M)&=\lmulti I_1=\left[\frac{1}{3},3\right), I_2=\left[\frac{2}{3},6\right), I_3=\left[0,9\right)\rmulti
        \\
        \barc(N)&=\lmulti J_1=\left[-\frac{2}{3},3\right), J_2=\left[-\frac{1}{3},6\right), J_3=\left[0,9\right)\rmulti.
    \end{align*}
    Observe that $\bott(\barc(M),\barc(N))=1$.
    Hence, by the isometry theorem %
    \cite[Theorem 3.4]{lesnick2015theory} and \Cref{thm:dedi=dint}, we have  $\dedi^{\mathsf{Grid}_1}(M,N)=d_I(M,N)=\bott(\barc(M),\barc(N))=1$. Therefore, it suffices to show that no length-$2$ $\mathsf{Grid}_1$-edit
sequence between $M$ and $N$ has cost $1$.

     Suppose that $M\xleftrightarrow{f_1} M_1 \xleftrightarrow{f_2} N$ is a $\mathsf{Grid}_1$-edit sequence with $\dist{f_1}+\dist{f_2}=1$. 
    By \Cref{thm:functoriality_of_barcode_1}, 
    we have the induced matchings $\chi_1:\barc(M)\not\to \barc(M_1)$ by $f_1$ and $\chi_2:\barc(M_1)\not\to \barc(N)$ by $f_2$ with $\cost(\chi_1)\le \dist{f_1}$ and $\cost(\chi_2)\le \dist{f_2}$.
        Since $\cost(\chi_1), \cost(\chi_2)\le 1$ and all of the intervals in $\barc(M)$ and $\barc(N)$ have length $>3$, each of the intervals in $\barc(M)$ and $\barc(N)$ is matched by either $\chi_1$ or $\chi_2$. For $i=1,2,3$, denote by $\chi_1(I_i)$ the interval matched to $I_i$, and by $\chi_2^{-1}(J_i)$ the interval matched to $J_i$. 

    We claim that $\chi_1(I_i)=\chi_2^{-1}(J_i)$ for $i=1,2,3$. Fix $i$. Suppose, for contradiction, that $\chi_1(I_i)\neq\chi_2^{-1}(J_i)$. 
    Then $\chi_1(I_i)$ is either unmatched by $\chi_2$, or  matched to $J_{i'}$ for some $i'\neq i$.
    In the former case,  \[1=\dist{f_1}+\dist{f_2}\geq \cost(\chi_1)+\cost(\chi_2)\ge d_\infty(I_i,\chi_1(I_i))+c(\chi_1(I_i))\ge c(I_i)> 1,\] a contradiction. 
    In the latter case,
    \[1=\dist{f_1}+\dist{f_2}\geq \cost(\chi_1)+\cost(\chi_2)\ge d_\infty (I_i,\chi_1(I_i))+d_\infty (\chi_1(I_i),J_{i'})\ge d_\infty(I_i,J_{i'})>1,\] a contradiction.

    Hence, we have that $\chi_1(I_i)=\chi_2^{-1}(J_i)$ for $i=1,2,3$. 
  For each $i$, let $\x_i$ be the left-endpoint of the interval $\chi_1(I_i)$. Consider the following cases:
    \begin{itemize}
        \item (Case 1) The given $\mathsf{Grid}_1$-edit sequence is $M_0\xleftarrow{f_1} M_1 \xrightarrow{f_2} M_2.$\vspace{3mm}
        
        As $\barc(M_0)=f_1\barc(M_1)$ and $\chi_1$ is a matching induced from $f_1$, we have $f_1(\chi_1(I_i))=I_i$ for each $i$. If we had $\x_1 \le \x_3$, by the monotonicity of $f_1$, $\frac13 = f(\x_1) \le f(\x_3)=0$, hence a contradiction. So we have $\x_1>\x_3$. Similarly, as $\barc(M_2)=f_2\barc(M_2)$ and $\chi_2$ is a matching induced from $f_2$, we also have $\x_1<\x_3$, which is a contradiction.
        
        \item (Case 2) The given $\mathsf{Grid}_1$-edit sequence is $M_0\xrightarrow{f_1} M_1 \xleftarrow{f_2} M_2.$\vspace{3mm}
        
            As $\barc(M_1)=f_1\ \barc(M_0)$ and $\chi_1$ is a matching induced from $f_1$, we have $f_1(I_i)=\chi_1(I_i)$ for each $i$. In particular, $\x_1=f\left(\frac13\right)$, $\x_2=f_1\left(\frac23\right)$, and $\x_3=f_1\left(0\right)$. By monotonicity of $f_1$, we have the inequality $\x_2\ge \x_1\ge \x_3$. Similarly, as $\barc(M_1)=f_2\barc(M_2)$ and $\chi_2$ is induced from $f_2$, we also have $\x_3\ge \x_2\ge \x_1$. Therefore $\x_1=\x_2$, and 
        \[1\ge \cost(\chi_1)+\cost(\chi_2)\ge d_\infty(I_2,\chi_1(I_2))+d_\infty(\chi_2^{-1}(J_1),J_1)\ge \abs{\frac23-\x_2}+\abs{\x_1-\left(-\frac23\right)}\ge\frac43,\] which is a contradiction.
    \end{itemize}
    If the given $\mathsf{Grid}_1$-edit sequence is either   $M_0\xrightarrow{f_1} M_1 \xrightarrow{f_2} M_2$ or $M_0\xleftarrow{f_1} M_1 \xleftarrow{f_2} M_2$, then we can turn it into one of the above two cases by letting one of $f_1$ and $f_2$ be an identity morphism in the category $\mathsf{Grid}_1$. 
    Hence, in all cases, we obtain a contradiction. Therefore, there cannot be a $\mathsf{Grid}_1$-edit sequence of length 2 between $M$ and $N$ that attains the value $1=\dedi^{\mathsf{Grid}_1}(M,N)$.
\end{proof}

\end{document}